\documentclass[a4paper]{amsart}
\usepackage[utf8]{inputenc}
\usepackage{amsfonts}
\usepackage{amsmath,latexsym,amssymb,amsfonts, amsthm}
\usepackage{esint}
\usepackage{cite}
\usepackage{graphicx}
\usepackage{amscd}
\usepackage{color}
\usepackage{bm}             
\usepackage{enumerate}
\usepackage{aliascnt}
\usepackage[dvips]{epsfig}
\usepackage{psfrag}
\usepackage{enumitem}
\usepackage{epsfig}

\usepackage[
  hmarginratio={1:1},     
  vmarginratio={1:1},     
  textwidth=16cm,        
  textheight=21cm,
  heightrounded,          
]{geometry}

\usepackage{graphicx,color}
\usepackage[colorlinks]{hyperref}
\usepackage{cleveref}
\hypersetup{linkcolor=blue,citecolor=blue,filecolor=black,urlcolor=blue}
\usepackage{mathtools}

\renewcommand{\b}{\beta}

\newtheorem{theorem}{Theorem}[section]
\crefname{theorem}{theorem}{theorems}
\Crefname{theorem}{Theorem}{Theorems}

\newaliascnt{proposition}{theorem}
\newtheorem{proposition}[proposition]{Proposition}
\aliascntresetthe{proposition}
\crefname{proposition}{proposition}{propositions}
\Crefname{proposition}{Proposition}{Propositions}

\newaliascnt{corollary}{theorem}

\aliascntresetthe{corollary}
\crefname{corollary}{corollary}{corollaries}
\Crefname{corollary}{Corollary}{Corollaries}

\newaliascnt{lemma}{theorem}
\newtheorem{lemma}[lemma]{Lemma}
\aliascntresetthe{lemma}
\crefname{lemma}{lemma}{lemmas}
\Crefname{lemma}{Lemma}{Lemmas}

\theoremstyle{definition}
\newaliascnt{remark}{theorem}
\newtheorem{remark}[remark]{Remark}
\aliascntresetthe{remark}
\theoremstyle{definition}
\crefname{remark}{remark}{remarks}
\Crefname{remark}{Remark}{Remarks}

\newaliascnt{example}{theorem}

\aliascntresetthe{example}
\theoremstyle{definition}
\crefname{example}{example}{examples}
\Crefname{example}{Example}{Examples}

\newaliascnt{definition}{theorem}

\aliascntresetthe{definition}
\theoremstyle{definition}
\crefname{definition}{definition}{definitions}
\Crefname{definition}{Definition}{Definitions}

\numberwithin{theorem}{section}

\numberwithin{equation}{section}

\newcommand{\beq}{\begin{equation}}
\newcommand{\eeq}{\end{equation}}

\newcommand{\R}{\mathbb{R}}
\newcommand{\N}{\mathbb{N}}

\newcommand{\mf}[1]{\mathbf{#1}}
\newcommand{\bs}[1]{\boldsymbol{#1}}

\newcommand{\pa}{\partial}
\renewcommand{\S}{\mathbb{S}}

\newcommand{\cH}{{\mathcal H}}

\newcommand{\cM}{{\mathcal M}}   
\newcommand{\cN}{{\mathcal N}}

\newcommand{\umin}{\mf u_{\mathrm{min}}}
\newcommand{\uminarg}[1]{u_{\mathrm{min},#1}}
\newcommand{\intRn}{\int_{\R^d}}
\newcommand{\norm}[1]{\|#1\|}
\newcommand{\oddpower}[2]{|#1|^{#2}#1}
\newcommand{\HoneradK}{H^1_{\mathrm{rad}}(\R^d,\R^K)}
\newcommand{\HoneK}{H^1(\R^d,\R^K)}

\DeclareMathOperator*{\argmin}{arg\,min}

\newcommand{\supp}{{\mathrm{supp}}}

\newcommand{\eps}{\varepsilon}

\usepackage{accents}
\newcommand{\ubar}[1]{\underaccent{\bar}{#1}}

\renewcommand{\epsilon}{\varepsilon}

\author[L. Giaretto and N. Soave]{Lorenzo Giaretto and Nicola Soave}\thanks{}
\address{Lorenzo Giaretto and Nicola Soave\newline \indent
 Dipartimento di Matematica ``Giuseppe Peano'', Universit\`a di Torino, \newline \indent
Via Carlo Alberto 10,
10123 Torino, Italy}
\email{lorenzo.giaretto@unito.it, nicola.soave@unito.it}

\title[On least energy solutions for a NLS system with $K$-wise interaction]{On least energy solutions for a nonlinear Schr\"odinger system with $K$-wise interaction}
\thanks{N.S. is partially supported by the PRIN Project no. 2022R537CS ``Nodal Optimization, NOnlinear elliptic equations, NOnlocal geometric problems, with a focus on regularity ($NO^3$)", funded by European Union - Next Generation EU within the PRIN 2022 program (D.D. 104 - 02$/$02$/$2022 Ministero dell\'Universit\`a e della Ricerca, Italy). Both authors are affiliated to the INDAM - GNAMPA group.\\
Declarations of interest: none. \\
Data availability: Data sharing not applicable to this article as no datasets were generated or analysed during the current study.}

\begin{document}

\begin{abstract}
In this paper we establish existence and properties of minimal energy solutions for the weakly coupled system
  \[
  \begin{cases}
     -\Delta u_i + \lambda_i u_i = \mu_i|u_i|^{Kq-2}u_i + \b|u_i|^{q-2}u_i\prod_{j\neq i}|u_j|^q         \quad\text{in }\R^d \\
     u_i \in H^1(\R^d), \end{cases}\qquad i=1,\dots, K, 
  \]
characterized by $K$-wise interaction (namely the interaction term involves the product of all the components). We consider both attractive ($\beta>0$) and repulsive cases ($\beta<0$), and we give sufficient conditions on $\beta$ in order to have least energy fully non-trivial solutions, if necessary under a radial constraint. We also study the asymptotic behavior of least energy fully non-trivial radial solutions in the limit of strong competition $\beta \to -\infty$, showing partial segregation phenomena which differ substantially from those arising in pairwise interaction models.
\end{abstract}

\maketitle

\section{introduction and main results}

In this paper we are interested in fully non-trivial minimal energy solutions of the system
  \begin{equation}\label{eq:general problem}
  \begin{cases}
     -\Delta u_i + \lambda_i u_i = \mu_i|u_i|^{Kq-2}u_i + \b|u_i|^{q-2}u_i\prod_{j\neq i}|u_j|^q         \quad\text{in }\R^d \\
     u_i \in H^1(\R^d), \end{cases}\qquad i=1,\dots, K, 
  \end{equation}
  for some integers $d\geq 2$ and $K\geq 3$, and real numbers $\lambda_i>0$, $\mu_i>0$ and $\b\in \R$. Since we look for energy solutions, we assume that the exponent $q$ satisfies
  $1\leq q<2^*/K=2d/(K(d-2))$ when $d\geq 3$ and $1\leq q<\infty$ when $d=2$. 
    
The case $K=2$ has been the subject of extensive research in the last twenty years (especially for $q=2$) in view of its physical relevance in nonlinear optics, in particular in the study of propagation of beams with multiple states in Kerr-like photorefractive media; several results regarding existence, non-existence, multiplicity, asymptotic behavior of the solutions are now available, also in the critical case. A comprehensive list is beyond the aim of this paper. We refer the interested reader to \cite{AmCo, CheZou, KoPiVa, Maia2006, Mandel2014cooperative, Mandel2014repulsive, LinWei1, LiuWan, PiSo, PiSoTa, Soa, SoaveTavares2016, TerVer, WeiWet} and references therein. Differently from the previous contributions, where the interaction takes place pairwise (namely the interaction term is
\[
\b|u_i|^{q-2} u_i |u_j|^q, \quad \text{or} \quad \b \sum_{j \neq i} |u_i|^{q-2} u_i |u_j|^q \quad \text{for systems with more than $2$ components}),
\]
we consider a \emph{$K$-wise interaction term} involving the product of (suitable powers of) all the components in \eqref{eq:general problem}.

In order to understand the structure of the system, it is convenient to think of $\mf{u}=(u_1,\dots,u_K)$ as different densities interacting in $\R^d$. The fact that $\mu_i>0$ reflects the fact that particles of the same density are related by an attractive interaction. The term $\b|u_i|^{q-2}u_i\prod_{j\neq i}|u_j|^q$ accounts for the simultaneous interaction among all $K$ densities, and the sign of the coefficient $\beta$ determines whether such interaction is attractive ($\beta > 0$) or repulsive ($\beta < 0$). The larger the values of $|\mu_i|$ and $|\beta|$, the stronger the corresponding interactions. In particular, we are interested in the limit $\beta \to -\infty$, known as the \emph{strong competition limit}. System \eqref{eq:general problem} can be seen as a generalization of the $K = 2$ case, in which $K$ different densities interact simultaneously. From this perspective, it is worth mentioning that systems with multiple (i.e., beyond binary) interaction terms naturally arise in various physical contexts; see, for instance, \cite{CoCoOh, LQZ24, Pet14, Po} and references therein.

In this study, inspired in particular by \cite{AmCo, Maia2006, Mandel2014cooperative, Mandel2014repulsive, Sir}, we provide explicit ranges for $\b$, depending on the value of the other data $\lambda_i, \mu_i, q, K, d$, in which least energy fully non-trivial solutions exist or not. Furthermore, we describe the asymptotic behavior of least energy fully non-trivial solutions in the limit of strong competition.

Before stating our main results, it is convenient to introduce some notation and review more in details some of the aforementioned results. Solutions to \eqref{eq:general problem} are critical points of the Euler-Lagrange energy functional (or, more appropriately, action functional) $I_\beta$ associated with \eqref{eq:general problem}, defined on $H^1(\R^d,\R^K)$ as
\[
I_\b(\mf{u}):= \frac{1}{2} \int_{\R^d}\sum_{i=1}^K\left( |\nabla u_i|^2 + \lambda_i |u_i|^2\right) - \frac{1}{Kq} \int_{\R^d} \sum_{i=1}^K  \mu_i|u_i|^{Kq} - \frac{\b}{q}\intRn  \prod_{i=1}^K |u_i|^q.
\]
\emph{Least energy solutions}, also called \emph{ground states}, are defined as solutions having minimal energy in the set of all non-trivial solutions. This set is clearly non-empty, since system \eqref{eq:general problem} admits \emph{semi-trivial} solutions (namely non-trivial solutions with some trivial components) such as $\mf w_i$, for $i=1,\dots,K$, where
\begin{equation*}
	(\mf w_i)_j=\begin{cases}
		w_i & j=i, \\
		0 & j\neq i,
	\end{cases}, \qquad i,j\in \left\{1,\dots,K\right\}. 
\end{equation*}
and $w_i$ is the unique positive solution to
\begin{equation}\label{eq:schrodinger 1 componente}
    -\Delta w+\lambda_i w=\mu_i w^{Kq-1} \quad \text{ in }\R^d,
\end{equation}
see \cite{Kwong1989-ln}. Moreover, we will be also interested in \emph{least energy fully non-trivial} (resp. \emph{positive}) \emph{solutions}, namely solutions with all non-trivial (resp. all positive) components, having minimal energy in the set of solutions with this same property. 

Since $Kq>2$, the energy functional $I$ is unbounded from below, and hence, to find least energy solutions, it is natural to proceed by constrained minimization. A first choice for this purpose is the Nehari manifold
\begin{equation}\label{eq:nehari scalare}
    \cN_\b=\left\{\mf u=(u_1,\dots, u_K)\in H^1(\R^d,\R^K): \ \mf{u}\neq \mf{0}, \ I_\b'(\mf{u})[\mf{u}]=0\right\},
\end{equation}
which by definition contains all the solutions of \eqref{eq:general problem}. Moreover, it is well known that $\cN_\b$ is a natural constraint, in the sense that any critical point of $I_\b|_{\cN_\b}$ is in fact a free critical point of $I_\b$, and hence a solution to \eqref{eq:general problem}. Thus, if the infimum of $I_\b$ on $\cN_\b$ is achieved, any minimizer is a ground state. One of the main advantages of dealing with the constrained functional $I_\b|_{\cN_\b}$ consists in the fact that
\begin{equation}\label{eq:Ibeta on nehari}
	I_\b|_{\cN_\b}(\mf u)=\left(\frac{1}{2}-\frac{1}{Kq}\right)\int_{\R^d}\sum_{i=1}^K\left( |\nabla u_i|^2 + \lambda_i |u_i|^2\right),
\end{equation}
which is bounded from below and weakly lower semi-continuous. Therefore, the proof of the existence of a minimizer is rather simple. Since however we are interested in least energy fully non-trivial solutions, a delicate problem consists in determining whether a ground state is semi-trivial or fully non-trivial. 

In order to search for least energy fully non-trivial solutions, another possibility is to minimize $I_\b$ on 
$$\cM_\b=\left\{\mf u\in H^1(\R^d,\R^K): \ u_i\neq 0, \ \partial_i I_\b(\mf{u})[u_i]=0 \text{ for }i=1,\dots,K \right\},$$
and its radial counterpart - whose introduction will be justified later
$$\cM_\b^r=\left\{\mf u\in H^1_{\mathrm{rad}}(\R^d,\R^K): \ u_i\neq 0, \ \partial_i I_\b(\mf{u})[u_i]=0 \text{ for }i=1,\dots,K \right\}.$$
Note that, by definition, any minimizer of $I_\b$ on $\cM_\b^r$ or on $\cM_\b$ is fully non-trivial; nevertheless, dealing with manifolds of higher codimension, the minimization is more involved, and the fact that a minimizer solves \eqref{eq:general problem} is not straightforward. We point out that
\begin{equation*}
	\cM_\b^r\subset \cM_\b \subset \cN_\b,
\end{equation*}
and on all these sets $I_\b$ can be written as in \eqref{eq:Ibeta on nehari}. By definition, if the infimum of $I_\beta|_{\cM_\b}$ is attained, it is attained by a least energy fully non-trivial solution (not necessarily a ground state). 

If instead the infimum of $I_\beta|_{\cM_\b^r}$ is attained, then a minimizer is a \emph{least energy fully non-trivial radial solution}, that is a fully non-trivial radial solution having minimal energy among all the fully non-trivial radial solutions (but this may not be a least energy fully non-trivial solution).

The existence of semi-trivial solutions is a key feature of system \eqref{eq:general problem}, which makes the system \emph{weakly coupled}. As already anticipated, our study is strongly motivated by the fundamental results obtained in \cite{AmCo, Maia2006, Mandel2014cooperative, Mandel2014repulsive, Sir} concerning the weakly coupled systems with pairwise interaction
 \begin{equation}\label{eq:binary}
  \begin{cases}
     -\Delta u + \lambda_1 u = \mu_1|u|^{2q-2}u + \b|u|^{q-2} |v|^q u      &  \quad\text{in }\R^d \\
     -\Delta v + \lambda_2 v = \mu_2|v|^{2q-2}v + \b|u|^{q} |v|^{q-2} v    &  \quad\text{in }\R^d \\
     u,v \in H^1(\R^d). \end{cases}
  \end{equation}
For this problem, it is known that there exists a sharp threshold $\bar \beta \ge 0$ such that:
\begin{itemize}
\item[($i$)] If $\beta>0$, then \eqref{eq:binary} has a non-negative least energy solution.
\item[($ii$)] If $\beta>\bar \beta$, then any least energy is fully non-trivial (thus, up to a change of sign each component is strictly positive in $\R^d$).
\item[($iii$)] If $0<\beta<\bar \beta$, then any least energy solution is semi-trivial.
\item[($iv$)] If $\beta=\bar \beta$, then there exists a least energy solution which is fully non-trivial.
\item[($v$)] It results that $\bar \beta=0$ for $q \in (1,2)$, while $\bar \beta >0$ when $q \ge 2$. 
\end{itemize}
Moreover, for the range $q \ge 2$ with $\beta \in (0, \bar \beta)$, it is known that there exists $\ubar{\beta} \in (0,\bar \beta]$ such that for $\beta<\ubar{\beta}$ system \eqref{eq:binary} has a least energy fully non-trivial solution (which is not a ground state) obtained by minimizing $I_\beta$ on $\cM_\b$. Moreover, such solution is strictly positive in $\R^d$ and radially symmetric. The precise characterization of $\ubar{\beta}$ is not completely settled however (we refer to \cite{IkoTan, Sir, WeZhZo} and references therein for more details).

Finally, in case $\beta<0$, the infimum of $I_\beta$ on $\cM_\b$ is not achieved, while problem \eqref{eq:binary} has a least energy positive radial solution $(u_\b, v_\b)$, obtained by minimizing $I_\beta$ on $\cM_\beta^r$. In the limit of strong competition it results that $(u_\b,v_\b) \to (\bar u, \bar v)$ strongly in $H^1(\R^d)$, where $\bar u$ and $\bar v$ are \emph{fully segregated}, i.e. $\bar u \, \bar v \equiv 0$ in $\R^d$, and $\bar u-\bar v$ is a least energy sign-changing solution to 
\[
-\Delta w +\lambda_1 w_+-\lambda_2 w_- = \mu_1 w_+^{2q-1} -\mu_2 w_-^{2q-1} \quad \text{in $\R^d$}
\] 
(see \cite{CoTeVe, NoTaTeVe, SoTaTeZi} and references therein).

We aim at understanding whether the generalization of these results for system \eqref{eq:general problem} is possible, and how the different type of interaction has an impact on the behavior of the solutions (in particular in the strong competition limit $\beta \to -\infty$). From this perspective, as we shall see rigorously, the simultaneous interaction of all the $K$ components leads to a \emph{partial segregation phenomenon}, namely $\bar u_1 \cdots \bar u_K \equiv 0$ in $\R^d$: thus, positivity sets of different components can overlap, but at least one component must vanish at any point. Partial segregation models are still not much understood, and only few results are available. We refer to \cite{CaRo, BoBuFo} for non-variational systems with symmetric interaction, and to \cite{SoTe1, SoTe2} for variational problems related to harmonic maps into singular spaces.

\subsection{Notation and main results} Let us fix some notation:
\begin{itemize}
    \item We denote by $\norm{\cdot}_i$ the equivalent norm on $H^1(\R^d)$ defined as
    \begin{equation*}
        \norm{u}_i:=\left(\intRn |\nabla u|^2+\lambda_i |u|^2\right)^\frac{1}{2};
    \end{equation*}
    \item We denote by $|\cdot|_p$ the standard norm on $L^p(\R^d)$, and by $|\cdot|_{p,i}$ the equivalent norm including the coefficient $\mu_i$, that is
    \begin{equation*}
        |u|_{p,i}:=\left(\intRn\mu_i|u|^p\right)^\frac{1}{p}.
    \end{equation*}
    \item We set
    \begin{equation}\label{min values}
    	\ell_\b := \inf_{\cN_\b} I, \quad k_\b := \inf_{\cM_\b} I, \quad k_\b^r := \inf_{\cM_\b^r} I.
    \end{equation}
    \item Finally, we denote by
    \begin{equation}\label{c_p}
    c_{p,\lambda,\mu}:=\inf_{u\in H^1(\R^d) \setminus \{0\}} \frac{\int_{\R^d}|\nabla u|^2+\lambda u^2}{\left(\int_{\R^d} \mu |u|^p\right)^{\frac{2}{p}}},
    \end{equation}
    and by
    \begin{equation}\label{eq:definition of S}
    \bar S:= \min \left\{c_{Kq, \lambda_i,\mu_i}: \ i\in \{1,\dots,K\}\right\}.
    \end{equation}
\end{itemize}
With these notations, we can write
\begin{align*}
	\cN_\b&=\left\{\mf u\in H^1(\R^d,\R^K): \ \mf{u}\neq \mf{0},\  \sum_{i=1}^K \norm{u_i}_i^2= \sum_{i=1}^K |u_i|_{Kq,i}^{Kq}+K\b |\prod_{i=1}^K u_i|_q^q\right\}, \\
	\cM_\b&=\left\{\mf u\in H^1(\R^d,\R^K): \ u_i\neq 0, \ \norm{u_i}_i^2=|u_i|_{Kq,i}^{Kq}+\b|\prod_{i=1}^K u_i|_q^q, \quad i=1,\dots,K \right\}, \\
	\cM_\b^r&=\left\{\mf u\in H^1_{\mathrm{rad}}(\R^d,\R^K): \ u_i\neq 0, \ \norm{u_i}_i^2=|u_i|_{Kq,i}^{Kq}+\b|\prod_{i=1}^K u_i|_q^q, \quad i=1,\dots,K \right\}.
\end{align*}
Moreover, we point out that
\[
\|w\|_i^2 \ge \bar S |w|_{Kq,i}^2 \quad \forall w \in H^1(\R^d), \ \forall i =1,\dots,K.
\]

\medskip

We now present the main results of the paper. The first one deals with the minimization on $\cN_\b$ defined in \eqref{eq:nehari scalare}.

\begin{theorem}\label{thm:beta grande}
    For all $\b>0$, the minimal value $\ell_\beta$ defined in \eqref{min values} is achieved, and $\ell_\b>0$. Let
    \begin{equation}
        \bar{\b}:=\inf\left\{\frac{\bar S^{-\frac{Kq}{2}}\left(\sum_{i=1}^K \norm{u_i}_i^2\right)^\frac{Kq}{2}-\sum_{i=1}^K|u_i|_{Kq,i}^{Kq}}{K|\prod_{i=1}^K u_i|_q^q}: \ \mf{u}\in H^1(\R^d,\R^K),\  \prod_{i=1}^K u_i\not\equiv 0\right\},
    \end{equation}
    where $\bar S$ is defined in \eqref{eq:definition of S}. Then:
    \begin{enumerate}[label=\arabic*)]
        \item if $\b<\bar{\b}$, then any ground state is a semi-trivial solution with only one non-trivial component;
        \label{item: beta<beta grande}
        \item if $\b>\bar\b$, then any ground state is fully non-trivial, and (up to a change of sign of some components, and of a translation) is a positive, radial solution to \eqref{eq:general problem}.
        \label{item:beta>beta grande}
    \end{enumerate}
\end{theorem}

\begin{remark}\label{rmk:equality of infima}
	Note that whenever $\mf u$ is a positive radial solution of \eqref{eq:general problem}, $\mf u\in \cM_\b^r$. Then for $\b>\bar \b$
	\begin{equation*}
		\ell_\b=k_\b=k_\b^r.
	\end{equation*} 
\end{remark}
The value $\bar\b =\bar{\b}(\bs{\lambda},\bs{\mu},q,d,K)$ is thus an explicit, sharp threshold for the nature of the energetically convenient configuration. Such a result is a natural generalization of what is known in the case $K=2$, and can be expected by the following intuition: when $\b$ is positive, the interaction term in the functional $I_\b$ is negative and hence contributes to a decrease of the energy, while the positive contribution comes from the $H^1$-norms. It is then natural that, for small values of $\b$ it is energetically convenient to have only one non-zero component reducing the $H^1$-norms which are the leading term; for higher values of $\b$ the interaction term plays a leading role and it becomes convenient to have a positive interaction term, given by non-zero components.

\medskip
Although explicit, the value $\bar\b$ by definition could in principle be negative and even negatively divergent. The next statement ensures that this cannot happen.

\begin{proposition}\label{thm:stime dal basso beta segnato}
    For $K\geq 3$ and $1\leq q<\frac{2d}{K(d-2)}$ if $d\geq 3$, $1\leq q<+\infty$ if $d=2$, we have that $\bar \beta>0$.
\end{proposition}

\begin{remark}\label{rmk: stima dal basso}
Numerical estimates suggest that in fact 
    \begin{equation*}
        \bar\b \geq \left(\prod_{i=1}^K\mu_i\right)^\frac{1}{K}\left(K^{\frac{Kq}{2}-1}-1\right)>0.
    \end{equation*}
\end{remark}


Within the situation depicted by \Cref{thm:beta grande}, the threshold $\b=\bar\b$ is not considered. It is natural to wonder what happens in this limit case. We can give a partial answer under the assumption $q>2$:

\begin{theorem}\label{thm:beta=beta segnato}
	If $q>2$, then $\ell_{\bar\b}$ is attained by a fully non-trivial ground state.
\end{theorem}
\begin{remark}
	It could still happen that a semi-trivial solution $\mf w_i$ attains $\ell_{\bar\b}$.
\end{remark}

We now turn our attention to the case $\b>0$ \emph{small}, that is, the range of values of $\b$ where ground states are semi-trivial. In this regime, it is natural to look for least energy positive solutions via the minimization of $I_\beta$ on $\cM_\b$, since for $\b>0$ the system is cooperative, any positive solution is radially symmetric with respect to a point, thanks to \cite{Busca2000}. This motivates us to directly restrict ourselves on $\cM_\b^r$.
 
\begin{theorem}\label{thm:beta positivo}
There exists a $\ubar{\b}>0$ such that, if $q \ge 2$ and $0<\b<\ubar{\b}$, then
    \begin{equation*}
        k_\b^r=\inf_{\cM_\b^r} I_\b
    \end{equation*}
    is achieved by a least energy positive solution to \eqref{eq:general problem}.
\end{theorem}

The value $\ubar{\b} =\ubar{\b}(\bs{\lambda},\bs{\mu},q,d,K)$ will be explicitly defined in \eqref{eq:definizione ubar beta}, and in particular it is not obtained by any limit process (and is not necessarily a ``small" quantity). It is an interesting open problem to understand whether the same result holds for $1 \le q <2$. We believe that this is the case, but we faced some technical complications which we could not overcome: the case $K=2$ is simpler, as it reduces to the study of critical points of functions of two variables. For a general $K>2$, such a direct argument is much harder and we followed a different path based on the study of minimizing sequences on $\cM_\b^r$.

Moreover, we point out that we proved existence of a least energy positive solution by minimizing $I_\b$ on $\cM_\b^r$, but we could not directly address the minimization on $\cM_\b$. It would be interesting to understand whether $\inf_{\cM_\b} I_\b < \inf_{\cM_\b^r} I_\b$ or if they coincide (in the former case the first infimum cannot be attained, since otherwise there exists a non-negative minimizer, i.e. a positive solution to \eqref{eq:general problem}, and positive solutions are necessarily radial by \cite{Busca2000}).

\medskip

Finally, we focus on the competitive case $\b<0$. Also for this range \Cref{thm:beta grande} motivates the search for least energy positive solutions, either radial or not, by minimizing $I_\b$ on $\cM_\b$ or on $\cM_\b^r$. Note that for $\b<0$ the system is not cooperative, so that the main result in \cite{Busca2000} is not applicable. Hence, it is not natural to restrict the study on $\cM_\b^r$ at a first stage. However, we can prove that the minimization of $\cM_\b$ without the radial constraint is not fruitful.

\begin{theorem}\label{thm:non esistenza non radiale beta negativo}
    Let $\b<0$. Then
    \begin{equation*}
        k_\b=\inf_{\cM_\b}I_\b=\sum_{i=1}^K I_\b(\mf w_i),
    \end{equation*}
    where $\mf w_i$ is defined in \eqref{eq:w_i definition}, and $k_\b$ is not achieved.
\end{theorem}

On the other hand, by minimizing $I_\b$ on the smaller set $\cM_\b^r$, existence is recovered.
\begin{theorem}\label{thm:esistenza radiale beta negativo}
    Let $\b<0$. Then $k_\b^r=\inf_{\cM_\b^r} I_\b$ is achieved by a least energy fully non-trivial radial solution to \eqref{eq:general problem}, with all non-negative components. 
\end{theorem}

For $q \ge 2$, the components of any non-negative fully non-trivial solutions are in fact strictly positive in $\R^d$, by the strong maximum principle; it is not clear whether the same property holds for $1\le q<2$. 

Theorems \ref{thm:beta grande}-\ref{thm:esistenza radiale beta negativo} give a rather complete picture for the existence of ground states and of least energy positive solutions for $\beta$ positive and large, and for $\beta<0$. The case when $\beta$ is positive but ``small" is more delicate, and this is consistent to what happen for $K=2$ (again, we refer to \cite{WeZhZo} and references therein). The proofs are inspired in particular by \cite{LinWei1, Mandel2014cooperative, Mandel2014repulsive,SoaveTavares2016}, but several steps have to be modified in order to deal with the different interaction term.

As far as the properties of the solutions are concerned, the main result in \cite{Busca2000} implies that for $\beta>0$ any positive solution is radially symmetric and radially decreasing with respect to a point. In what follows we aim at understanding some properties of least energy radial positive solutions for $\beta<0$. In particular, we focus on the asymptotic behavior of the solutions achieving $k_\b^r$ as $\b \to -\infty$. The properties of the limit allow to understand the structure of the minimizers for $k_\b^r$ when $\beta$ is negative and large.

In the present setting, we introduce the ``limit set"
\begin{equation*}
    \cM_{-\infty}^r=\left\{\mf u\in H^1_{\mathrm{rad}}(\R^d,\R^K): \  \prod_{i=1}^K u_i\equiv 0, \ u_i\neq 0, \ \norm{u_i}_i^2=|u_i|_{Kq,i}^{Kq} \, \text{ for }i=1,\dots,K \right\}.
\end{equation*}
and the functional
\begin{equation*}
    I_{-\infty}(\mf{u})= \sum_{i=1}^K \bigg[ \frac{1}{2} \int_{\R^d}\left( |\nabla u_i|^2 + \lambda_i |u_i|^2\right) - \frac{1}{Kq} \int_{\R^d}   \mu_i|u_i|^{Kq} \bigg].
\end{equation*}

With this notation we can state the following:
\begin{theorem}\label{thm:asintotica beta -infinito}
    The level
    \begin{equation*}
        k_{-\infty}^r:=\inf_{\cM_{-\infty}^r} I_{-\infty}
    \end{equation*}
    is attained by a non-negative $\mf{u}$ solving 
    \begin{equation}\label{eq:eq schrodinger limite all'infinito}
        -\Delta u_i+\lambda_i u_i=\mu_i u_i^{Kq-1}, \quad u_i>0 \quad \text{ in } \{u_i\neq 0\}, \quad \prod_{i=1}^K u_i\equiv 0.
    \end{equation}
    Moreover, $k_\b^r \to k_{-\infty}^r$ as $\beta \to -\infty$, and for any sequence $(\mf{u}_\b)_{\b<0}$ of minimizers of $I_\beta|_{\mathcal{M}_\beta^r}$ there exists a subsequence $\b_n \to -\infty$ and a minimizer $\mf{u}_{-\infty}$ for $k_{-\infty}^r$ such that
    \begin{equation*}
        |\b_n| \int_{\R^d} \prod_{i=1}^K |u_{\b_n,i}|^q \to 0 \quad \text{ and }  \quad \mf{u}_{\b_n}\to\mf{u}_{-\infty} \quad \text{strongly in } H^1(\R^d,\R^K)
    \end{equation*}
    as $n \to \infty$.
\end{theorem}

The previous statement allows to characterize the limit problem. In our last result, we give a full description of the minimizers of $k_{-\infty}^r$. 

\begin{theorem}\label{thm: asy beh}
Let $\mf{v}$ be a minimizer for $k_{-\infty}^r$. Then, up to a change of sign in some components, there exist two indexes $i_1, i_2 \in \{1,\dots,K\}$ such that 
\begin{itemize}
\item[($i$)] $v_{i_1} v_{i_2} \equiv 0$ in $\R^d$;
\item[($ii$)] $v_j>0$ in $\R^N$ for every $j \neq i_1, i_2$, and is a ground state solution of
\begin{equation}\label{sc pb leps 1}
-\Delta v+\lambda_j v= \mu_j v^{Kq-1}, \quad v>0 \qquad \text{in $\R^d$}.
\end{equation}
\item[($iii$)] The function $w := v_{i_1}-v_{i_2}$ is a \emph{least energy sign-changing radial solution} to
\begin{equation}\label{lescs}
-\Delta w+\lambda_{i_1} w_+-\lambda_{i_2} w_- =\mu_{i_1} w_+^{Kq-1} -\mu_{i_2} w_-^{Kq-1} \qquad \text{in $\R^d$},
\end{equation}
namely a sign-changing radial solution having minimal energy among all the sign-changing radial solutions.
\end{itemize}
\end{theorem}

This theorem reveals that, in order to minimize $k_{-\infty}^r$, it is convenient that two components, say $v_1$ and $v_2$, segregate completely, thus ``paying" the partial segregation condition $\prod_i v_i \equiv 0$ for all the others. At this point, the remaining components are least energy solutions of the corresponding scalar problems, while $v_1$ is the positive part and $v_2$ is the negative part of a least energy sign-changing solution to \eqref{lescs}. As already pointed out, in view of the convergence $\mf{u}_\b \to \mf{u}$ in \Cref{thm:asintotica beta -infinito}, this scenario describes approximately the shape of least energy radial positive solutions to \eqref{eq:general problem} for $\beta$ negative and large. 

We point out that Theorem \ref{thm:asintotica beta -infinito} and in particular \ref{thm: asy beh} reflect a partial segregation phenomenon, and have no counterpart in pairwise interaction systems (whose asymptotic limit is characterized by full segregation).

\begin{remark}
We conclude this introduction with further comments.

1) So far, we always assumed that $d \ge 2$. On the other hand, the problem can be also studied in the $1$-dimensional case. Some of the previous results can be extended straightforwardly. Others present substantial differences (as already pointed out in \cite{Mandel2014repulsive} for the case of binary interaction). We will not give detailed proofs, for the sake of brevity, but we refer the interested reader to Remark \ref{rmk: on d=1} at the end of the paper for a comprehensive discussion. 

Still regarding the dimension, we observe that since it must be $3 \le K < 2d/(d - 2)$, we can equivalently obtain some limitations on $d$ as a function of $K$.

2) It would also be interesting to consider the case $2/3<q<1$. Indeed, in this range the energy function $I_\beta$ is still unbounded from below, since the problem is superlinear. On the other hand, the interaction term in the system is sublinear with respect to each component $u_i$, and hence the functional is not $C^1$ in the full space $H^1(\R^d,\R^K)$.

3) Finally, we point out that we considered the problem in $\R^d$, but most of the results remain true if we consider the system in a bounded domain $\Omega$, prescribing homogeneous Dirichlet boundary conditions. In such a case, obviously it is not possible to restrict working with radial functions (but not even necessary, since the extra compactness given by the radial constraint is granted by the boundedness of $\Omega$). The only result which is affected by this change is Theorem \ref{thm: asy beh}, since in the proof we used some properties of the nodal sets of radial functions, which are $1$-dimensional in nature. The extension of such properties to non-radial functions arising as limits of \eqref{eq:general problem} when $\beta\to -\infty$ is not straightforward. One can probably adapt the arguments in \cite{SoTe1, SoTe2}, but the adaptation is certainly not trivial, and goes beyond the aim of this work.
\end{remark}

\subsection{Structure of the paper}
The paper is organized as follows. In \Cref{sec:preliminary}, we review some basic useful facts. \Cref{sec: proof beta grande} contains the proofs of the results regarding ground states, namely \Cref{thm:beta grande}, \Cref{thm:stime dal basso beta segnato} and \Cref{thm:beta=beta segnato}. The results regarding least energy positive solutions, Theorems \ref{thm:beta positivo}-\ref{thm:esistenza radiale beta negativo}, are proved in \Cref{sec:esistenza minimo radiale beta piccolo e negativo}. Finally, in Section \ref{sec:proof asymptotic} we address the asymptotic behavior of the solution, proving Theorems \ref{thm:asintotica beta -infinito} and \ref{thm: asy beh}.

\section{Preliminary results}\label{sec:preliminary}

In this section we recall some well-established facts about the scalar problem
\begin{equation}\label{eq:scalar schrodinger equation}
	\begin{cases}-\Delta u + \lambda u=\mu u^{p-1}, \quad u>0 \quad \text{in } \R^d\\ u \in H^1(\R^d) \end{cases} \qquad 2<p<2^*.
\end{equation}
It is well known than any solution is radially symmetric with respect to a point; moreover, the radial solution is unique up to translations (see \cite{GNN, Kwong1989-ln}). By scaling
\begin{equation*}
	w_{p,\lambda,\mu}(x)=\mu^{-\frac{1}{p-2}}\lambda^{\frac{1}{p-2}}w_p(\sqrt{\lambda}x),
\end{equation*} 
where $w_p=w_{p,1,1}$ is the unique positive solution of \eqref{eq:scalar schrodinger equation} with $\lambda=\mu=1$. These solutions can be variationally characterized as follows: they are, equivalently, both constrained minimizers of the energy \begin{equation*}
	E_{p,\lambda,\mu}(u):=\frac{1}{2}\int_{\R^d} \left(|\nabla u|^2+\lambda u^2 \right)- \frac{\mu}{p}\int_{\R^d} |u|^p, 
\end{equation*}
on the Nehari manifold $\cN_{p,\lambda,\mu}=\left\{u\in H^1(\R^d)\setminus \left\{0\right\}: \ E_{p,\lambda,\mu}'(u)[u]=0\right\}$; and minimizers of the quotient
\begin{equation*}
	R_{p,\lambda,\mu}(u):= \frac{\int_{\R^d}|\nabla u|^2+\lambda u^2}{\left(\int_{\R^d} \mu |u|^p\right)^{\frac{2}{p}}}
\end{equation*}
over $H^1(\R^d) \setminus \{0\}$. Recalling the definition of $c_{p,\lambda,\mu}$ given in \eqref{c_p}, by uniqueness and scaling we have that
\begin{equation*}
	c_{p,\lambda,\mu}=\frac{\intRn\lambda^{1+\frac{2}{p-2}}\mu^{-\frac{2}{p-2}}\left(|\nabla w_p(\sqrt{\lambda}x)|^2+w_p(\sqrt{\lambda}x)^2\right)dx}{\left(\intRn \mu^{1-\frac{p}{p-2}}\lambda^{\frac{p}{p-2}}|w_p(\sqrt{\lambda}x)|^p dx\right)^{\frac{2}{p}}}=\lambda^{1-\frac{d}{2}+\frac{d}{p}}\mu^{-\frac{2}{p}}c_p.
\end{equation*} 
Moreover, since $w_{p,\lambda,\mu}\in \cN_{p,\lambda,\mu}$,
\begin{align*}
	E_{p,\lambda,\mu}(w_{p,\lambda,\mu})&=\left(\frac{1}{2}-\frac{1}{p}\right)\left(\int_{\R^d}|\nabla w_{p,\lambda,\mu}|^2+\lambda w_{p,\lambda,\mu}^2\right)\\
	&=\left(\frac{1}{2}-\frac{1}{p}\right)\frac{\left(\int_{\R^d}|\nabla w_{p,\lambda,\mu}|^2+\lambda w_{p,\lambda,\mu}^2\right)^{\frac{p}{p-2}}}{\left(\int_{\R^d} \mu |w_{p,\lambda,\mu}|^p\right)^{\frac{2}{p-2}}},
\end{align*}
and the energy of $w_{p,\lambda,\mu}$ can be expressed in terms of $c_{p,\lambda,\mu}$ as
\begin{equation*}
	E_{p,\lambda,\mu}(w_{p,\lambda,\mu})=\left(\frac12-\frac1p\right)c_{p,\lambda,\mu}^{\frac{p}{p-2}}.
\end{equation*}

The scalar problem plays a crucial role in studying system \eqref{eq:general problem}, since scalar solutions define the following semi-trivial solutions (whose energy will often be compared with the one of any fully non-trivial solution): with parameters $\mu_i$, $\lambda_i$, $\b$ as in the general setting \eqref{eq:general problem}, we denote by
\begin{equation}\label{eq:w_i definition}
	(\mf w_i)_j=\begin{cases}
		w_{Kq,\lambda_i,\mu_i} & j=i, \\
		0 & j\neq i,
	\end{cases}, \qquad i,j\in \left\{1,\dots, K\right\}. 
\end{equation}
In particular, letting
\begin{equation*}
	\bar i= \argmin_{i=1,\dots,k} \ c_{Kq,\lambda_i,\mu_i},
\end{equation*}
we define the semi-trivial solution with lowest energy as
\begin{equation}\label{eq:w bar definition}
	\mf{\bar w}:=\mf w_{\bar i}.
\end{equation}
According to \eqref{eq:definition of S}, for all $\b\in\R$
\begin{equation}\label{eq:S and energy of wbar}
	I_\b(\mf {\bar w})=\left(\frac{1}{2}-\frac{1}{Kq}\right) \bar S^{\frac{Kq}{Kq-2}}.
\end{equation}
We emphasize that there exist also semi-trivial solutions to \eqref{eq:general problem} with more than one non-trivial component, such as $(w_{Kq,\lambda_1,\mu_1}, w_{Kq,\lambda_2,\mu_2}, 0,\dots,0)$. However, any such solution is not a ground state. For instance, by thinking at the previous example, we see that
\[
\begin{split}
I_\beta(w_{Kq,\lambda_1,\mu_1}, w_{Kq,\lambda_2,\mu_2}, 0,\dots,0) &= E_{Kq,\lambda_1,\mu_1}(w_{Kq,\lambda_1,\mu_1}) + E_{Kq,\lambda_2,\mu_2}(w_{Kq,\lambda_2,\mu_2})\\
& > \min\{E_{Kq,\lambda_1,\mu_1}(w_{Kq,\lambda_1,\mu_1}), E_{Kq,\lambda_2,\mu_2}(w_{Kq,\lambda_2,\mu_2})\} \\
&=\min\{I_\b (w_{Kq,\lambda_1,\mu_1}),I_\b(w_{Kq,\lambda_2,\mu_2})\},
\end{split}
\]
and the same argument shows that, among the semi-trivial solutions, those with only one non-trivial component have the lowest possible energy.

\section{Existence and properties of ground states}\label{sec: proof beta grande}
The section is devoted to the proof of \Cref{thm:beta grande}, \Cref{thm:stime dal basso beta segnato} and \Cref{thm:beta=beta segnato}.

At first, we aim at studying the minimization of $I_\b$ on the Nehari manifold $\cN_\b$, namely the existence of ground states. The next proposition shows that this problem has a solution, for any $\b>0$.

\begin{proposition}\label{prop:existence of minimizers on nehari vectorial}
	Let $\b>0$. Then
	\begin{equation*}
		\ell_\b=\inf_{\cN_\b} I_\b
	\end{equation*}
	is achieved by a non-negative radial solution to \eqref{eq:general problem}.
\end{proposition}

The proof of this fact has been established in \cite[Theorem 2.1]{Maia2006} for the binary interaction and can be extended in a straightforward way to our context (we notice that the symmetry result in \cite{Busca2000} involved to prove the radial symmetry still applies, since for $\beta>0$ system \eqref{eq:general problem} is cooperative).

It is well known that the Nehari manifold is a natural constraint, in the sense that constrained critical points for $I_\b|_{\cN_\b}$ are in fact free critical points of $I_\b$, and hence solutions to \eqref{eq:general problem}. In particular, this is true for any minimizer, whose existence is ensured by \Cref{prop:existence of minimizers on nehari vectorial}. Then, to conclude the proof of \Cref{thm:beta grande}, we compare the energy of the semi-trivial solutions with the one of fully non-trivial elements in $\cN_\b$.

\begin{proof}[Proof of \Cref{thm:beta grande}, \ref{item: beta<beta grande}]
We want to show that the energy of any fully non-trivial function is strictly larger than $I_\b(\bar{\mf{w}})=\min_{i\in \{1,\dots,K\}}I_\b(\mf{w}_i)$. Let $\b<\bar{\b}$ and assume by contradiction that there exists a fully non-trivial ground state, namely a function $\mf{u}\in \cN_\b$ such that $u_i\neq 0$ for $i=1,\dots,K$ and $I_{\b}(\mf{u}) = \ell_\b$. It is not restrictive to suppose that $u_i \ge 0$ in $\R^d$ for every $i$, and, by the strong maximum principle, we have in fact that $u_i>0$ in $\R^d$ for every $i$. Therefore, $\prod_{i=1}^K u_i > 0$ in $\R^d$ as well and, by definition of $\bar{\b}$,
	\begin{equation*}
		\b<\frac{\bar S^{-\frac{Kq}{2}}\displaystyle \Big(\sum_{i=1}^K\norm{u_i}_i^2\Big)^{\frac{Kq}2}-\sum_{i=1}^K|u_i|_{Kq,i}^{Kq}}{K|\prod_{i=1}^K u_i|_q^q}.
	\end{equation*}
	Recalling \eqref{eq:S and energy of wbar}, it follows that
	\begin{align*}
		I_\b(\mf u) & =\left(\frac{1}{2}-\frac{1}{Kq}\right)\sum_{i=1}^K\norm{u_i}_i^2=\left(\frac{1}{2}-\frac{1}{Kq}\right) \left(\frac{\left(\sum_{i=1}^{K}\norm{ u_i}_{i}^2\right)^{\frac{Kq}{2}}}{\sum_{i=1}^{K}|u_i|_{Kq,i}^{Kq}+K\b |\prod_{i=1}^K u_i|_q^q}\right)^\frac{2}{Kq-2}\\
		&>\left(\frac{1}{2}-\frac{1}{Kq}\right)\bar S^{\frac{Kq}{Kq-2}}=I_\b(\mf {\bar w})\geq \ell_\b,
	\end{align*}
	contradicting the minimality of $\mf u$. Since by \Cref{prop:existence of minimizers on nehari vectorial} the infimum $\ell_\b$ is achieved, it must be achieved by $\mf {\bar w}$ only.
\end{proof}

\begin{proof}[Proof of \Cref{thm:beta grande}, \ref{item:beta>beta grande}]
	If $\b>\bar{\b}$, then by definition of $\bar{\b}$ there exists $\mf u \in H^1(\R^d,\R^K)$ such that 	\begin{equation}\label{eq:inequality 1}
		K\b|\prod_{i=1}^K u_i|_q^q> \bar S^{-\frac{Kq}{2}}\left(\sum_{i=1}^K\norm{u_i}_i^2\right)^{\frac{Kq}{2}}-\sum_{i=1}^K |u_i|_{Kq,i}^{Kq}.
	\end{equation}
	We claim that the energy of this element is strictly less than $I_\b(\mf{\bar w})$. Combined with \Cref{prop:existence of minimizers on nehari vectorial}, and recalling that $\mf{\bar w}$ has the lowest energy among all semi-trival solutions, the existence of a fully non-trivial  and non-negative ground state $\umin$ is proved, together with the fact that all ground states are necessarily fully non-trivial. By the strong maximum principle, each component of $\umin$ is strictly positive in $\R^d.$
	
	To prove the claim, it is not restrictive to assume that $\mf u\in \cN_\b$. Indeed, if not, let $t>0$ be such that $t\mf u\in \cN_\b$, that is,
	\begin{equation*}
		t=\left(\frac{\sum_{i=1}^K \norm{u_i}_i^2}{\sum_{i=1}^K |u_i|_{Kq,i}^{Kq}+K\b|\prod_{i=1}^K u_i|_q^q}\right)^\frac{1}{Kq-2}.
	\end{equation*}
	By homogeneity, \eqref{eq:inequality 1} holds for $t\mf u$ as well. Then, since $\mf u\in \cN_\b$ and by \eqref{eq:inequality 1} and \eqref{eq:S and energy of wbar},
	\begin{align*}
		I_\b(\mf u)&=\left(\frac{1}{2}-\frac{1}{Kq}\right) \left(\frac{\left(\sum_{i=1}^{K}\norm{ u_i}_{i}^2\right)^{\frac{Kq}{2}}}{\sum_{i=1}^{K}|u_i|_{Kq,i}^{Kq}+K\b |\prod_{i=1}^K u_i|_q^q}\right)^\frac{2}{Kq-2} \\
		&<\left(\frac{1}{2}-\frac{1}{Kq}\right) \bar S^{\frac{Kq}{Kq-2}}=I_\b(\mf{\bar w}), 
	\end{align*}
which is the desired estimate.
\end{proof}

Now we prove the positivity of $\bar \beta$.

\begin{proof}[Proof of \Cref{thm:stime dal basso beta segnato}]
	Let $\mf u\in H^1(\R^d)$ be such that $\prod_{i=1}^K u_i\neq 0$. Then
	\begin{equation*}
		|\prod_{i=1}^K u_i|_q^q\leq \prod_{i=1}^K |u_i|_{Kq}^q = \prod_{i=1}^K \mu_i^{-\frac{1}{K}}|u_i|_{Kq,i}^q.
	\end{equation*}
	We define, for a fixed $\mf u$, 
	\begin{equation*}
		s_i=\frac{|u_i|_{Kq,i}}{|u_K|_{Kq,K}}, \quad i=1,\dots, K-1,
	\end{equation*}
	and for notational purpose we set $s_K=1$.
	Then, by the Sobolev and H\"older inequality,
	\begin{align*}
		\norm{u_i}_i^2&\geq \bar S|u_i|_{Kq,i}^2=\bar S s_i^2|u_K|_{Kq,K}^2, \\
		|\prod_{i=1}^K u_i|_q^q&\leq \left(\prod_{i=1}^K s_i^q \mu_i^{-\frac{1}{K}}\right) |u_K|_{Kq,K}^{Kq}.
	\end{align*}
	As a consequence
	\begin{align*}
		\frac{\bar S^{-\frac{Kq}{2}}\left(\sum_{i=1}^K \norm{u_i}_i^2\right)^\frac{Kq}{2}-\sum_{i=1}^K|u_i|_{Kq,i}^{Kq}}{K|\prod_{i=1}^K u_i|_q^q}&\geq \frac{\bar S^{-\frac{Kq}{2}}\left(\bar S|u_K|_{Kq,K}^2(\sum_{i=1}^K s_i^2)\right)^\frac{Kq}{2}-|u_K|_{Kq,K}^{Kq}\left(\sum_{i=1}^K s_i^{Kq} \right)}{K |u_K|_{Kq,K}^{Kq} \prod_{i=1}^K s_i^q \mu_i^{-\frac{1}{K}}} \\
		&\geq \frac{(\sum_{i=1}^K s_i^2)^{\frac{Kq}{2}}-\sum_{i=1}^K s_i^{Kq}}{K \prod_{i=1}^K s_i^q \mu_i^{-\frac{1}{K}}}.
	\end{align*}
	Note that the last term depends on $\mf u$ only through the quantities $s_i$. Recalling the definition of $s_i$, we obtain
\begin{equation*}
		\bar{\b}\geq \left(\prod_{i=1}^K \mu_i\right)^{\frac{1}{K}}\inf_{s_1,\dots, s_{K-1}>0} \frac{(1+\sum_{i=1}^{K-1} s_i^2)^{\frac{Kq}{2}}-1-\sum_{i=1}^{K-1} s_i^{Kq}}{K\prod_{i=1}^{K-1} s_i^q}.
	\end{equation*}
	We now prove that
\begin{equation*}
	\inf_{s_1,\dots, s_{K-1}>0} \frac{(1+\sum_{i=1}^{K-1} s_i^2)^{\frac{Kq}{2}}-1-\sum_{i=1}^{K-1} s_i^{Kq}}{K\prod_{i=1}^{K-1} s_i^q}>0.
\end{equation*}
Since
\begin{equation*}
	(1+\sum_{i=1}^{K-1} s_i^2)^{\frac{Kq}{2}}>1+\sum_{i=1}^{K-1} s_i^{Kq} \quad \forall s_1,\dots,s_{K-1}>0,
\end{equation*}
clearly the infimum is non-negative. Passing to hyper-spherical coordinates, we write $\bs s=(s_1,\dots, s_{K-1})$ as
\begin{equation*}
	\bs s=\rho \bs \theta, \quad \bs \theta=(\theta_1,\dots,\theta_{K-1})\in \S^{K-2}_+, \quad \S^{K-2}_+=\left\{\bs \theta\in \S^{K-2}:\theta_i>0 \quad \forall i=1,\dots,K-1\right\},
\end{equation*}
so that
\begin{equation*}
	\inf_{s_1,\dots, s_{K-1}>0} \frac{(1+\sum_{i=1}^{K-1} s_i^2)^{\frac{Kq}{2}}-1-\sum_{i=1}^{K-1} s_i^{Kq}}{K\prod_{i=1}^{K-1} s_i^q}=\inf_{\substack{\bs \theta\in \S^{K-2}_+\\ \rho>0}} \frac{\left(1+\rho^2\right)^{\frac{Kq}{2}}-1-\rho^{Kq}\sum_{i=1}^{K-1}\theta_i^{Kq}}{\rho^{q(K-1)}\prod_{i=1}^{K-1}\theta_i^q}.
\end{equation*}
We split the analysis in the regions $\{\rho\leq 1\}$ and $\{\rho>1\}$.
Since $\sum_{i=1}^{K-1}\theta_i^{Kq}\leq \sum_{i=1}^{K-1}\theta_i^2=1$ and $\prod_{i=1}^{K-1}\theta_i^q\leq 1$, we have that
\begin{equation*}
	\frac{\left(1+\rho^2\right)^{\frac{Kq}{2}}-1-\rho^{Kq}\sum_{i=1}^{K-1}\theta_i^{Kq}}{\rho^{q(K-1)}\prod_{i=1}^{K-1}\theta_i^q}\geq \frac{\left(1+\rho^2\right)^{\frac{Kq}{2}}-1-\rho^{Kq}}{\rho^{q(K-1)}},
\end{equation*}
the last term being positive for any $\rho>0$.
Noticing that under the assumptions on $q$ and $K$ it holds
\begin{equation*}
	\lim_{\rho\to 0^+} \frac{\left(1+\rho^2\right)^{\frac{Kq}{2}}-1-\rho^{Kq}}{\rho^{q(K-1)}}>0,
\end{equation*}
this proves that
\begin{equation*}
	\inf_{\substack{\bs \theta\in \S^{K-2}_+\\ 0<\rho\leq 1}} \frac{\left(1+\rho^2\right)^{\frac{Kq}{2}}-1-\rho^{Kq}\sum_{i=1}^{K-1}\theta_i^{Kq}}{\rho^{q(K-1)}\prod_{i=1}^{K-1}\theta_i^q}>0.
\end{equation*}

The case $\rho\geq 1$ requires a slightly finer argument at infinity. We first observe that
\begin{equation}\label{eq:limite a rho inf}
	\lim_{\rho\to +\infty} \frac{\left(1+\rho^2\right)^{\frac{Kq}{2}}-1-\rho^{Kq}}{\rho^{q(K-2)}}>0.
\end{equation}
Then, we estimate
\begin{align*}
	\frac{\left(1+\rho^2\right)^{\frac{Kq}{2}} - 1 - \rho^{Kq} \sum_{i=1}^{K-1} \theta_i^{Kq}}{\rho^{q(K-1)} \prod_{i=1}^{K-1} \theta_i^q} &= \frac{(1+\rho^2)^{\frac{Kq}{2}} - 1 - \rho^{Kq}}{\rho^{(K-2)q}} \cdot \frac{1}{\rho^q \prod_{i=1}^{K-1} \theta_i^q}
	+ \frac{\rho^q \left(1 - \sum_{i=1}^{K-1} \theta_i^{Kq} \right)}{\prod_{i=1}^{K-1} \theta_i^q} \\
	&\geq C \frac{1}{\rho^q \prod_{i=1}^{K-1} \theta_i^q}
	+ \frac{\rho^q \left(1 - \sum_{i=1}^{K-1} \theta_i^{Kq} \right)}{\prod_{i=1}^{K-1} \theta_i^q} \\
	& \geq 2 \sqrt{C} \frac{\left(1 - \sum_{i=1}^{K-1} \theta_i^{Kq} \right)^{1/2}}{\prod_{i=1}^{K-1} \theta_i^q},
\end{align*}	
where we used Young's inequality and \eqref{eq:limite a rho inf}. We are left to prove that
\begin{equation*}
	\inf_{\bs \theta\in \S^{K-2}_+} \frac{\left(1 - \sum_{i=1}^{K-1} \theta_i^{Kq} \right)^{1/2}}{\prod_{i=1}^{K-1} \theta_i^q} >0.
\end{equation*}
Since $q\geq 1$, it is actually enough to prove that
\begin{equation*}
	\inf_{\bs \theta\in \S^{K-2}_+} \frac{\left(1 - \sum_{i=1}^{K-1} \theta_i^{K} \right)^{1/2}}{\prod_{i=1}^{K-1} \theta_i} >0.
\end{equation*}
The function is strictly positive for all $\bs \theta\in \S^{K-2}_+$ and tends to zero only if $\theta_i\to 1$ for some $i$ and $\theta_j\to 0$ for all $j\neq i$. As a consequence, it suffices to prove, without loss of generality, that
\begin{equation*}
	\lim_{\substack{\theta_1\to 1 \\ \sum_i\theta_i^2=1}} \frac{\left(1-\sum_{i=1}^{K-1}\theta_i^k\right)}{\prod_{i=1}^{K-1}\theta_i}>0.
\end{equation*}
If $\theta_2^2+\dots+\theta_{K-1}^2=r^2\to 0$, then $\theta_1=\sqrt{1-r^2}$, $\theta_2, \dots, \theta_{K-1}\leq r$ and
\begin{align*}
	\lim_{\substack{\theta_1\to 1 \\ \sum_i\theta_i^2=1}} \frac{\left(1-\sum_{i=1}^{K-1}\theta_i^K\right)}{\prod_{i=1}^{K-1}\theta_i}&\geq \lim_{r\to 0^+} \frac{\left(1-(1-r^2)^{\frac{K}{2}}-(K-2)r^K\right)^{\frac{1}{2}}}{\sqrt{1-r^2} \, r^{K-2}}\\
	&=\lim_{r\to 0^+}\frac{\left(\frac{K}{2}r^2-(K-2)r^K\right)^{\frac{1}{2}}}{\left(1-\frac{1}{2}r^2\right)r^{K-2}}=\lim_{r\to 0^+}C r^{3-K}>0,
\end{align*}
and the proof is concluded. \end{proof}

\begin{remark}
The lower estimate for $\bar \beta$ given in Remark \ref{rmk: stima dal basso} comes from the fact that
\begin{equation*}
	\inf_{s_1,\dots, s_{K-1}>0} \frac{(1+\sum_{i=1}^{K-1} s_i^2)^{\frac{Kq}{2}}-1-\sum_{i=1}^{K-1} s_i^{Kq}}{K\prod_{i=1}^{K-1} s_i^q}=K^{\frac{Kq}{2}-1}-1,
\end{equation*}
the infimum being computed through numerical optimization algorithms.
\end{remark}

We conclude the section with the proof of the partial result about the threshold case, $\b = \bar \b$.

\begin{proof}[Proof of \Cref{thm:beta=beta segnato}]
	We first prove that the function
	\begin{equation*}
		\b \mapsto \ell_\b
	\end{equation*}
	is non-increasing on $(0,+\infty)$. For any $\mf u\in \cN_{\b}$
		\begin{equation*}
		I_\b(\mf u)=\left(\frac{1}{2}-\frac{1}{Kq}\right)\sum_{i=1}^K \norm{u_i}_i^2=\left(\frac{1}{2}-\frac{1}{Kq}\right) \left(\frac{\left(\sum_{i=1}^{K}\norm{ u_i}_{i}^2\right)^{\frac{Kq}{2}}}{\sum_{i=1}^{K}|u_i|_{Kq,i}^{Kq}+K\b |\prod_{i=1}^K u_i|_q^q}\right)^\frac{2}{Kq-2}\eqcolon\mathcal{R}_\b(\mf u),
	\end{equation*}
	whence we deduce that
	\begin{equation}\label{st2061}
		\ell_\b=\inf_{\cN_\b} I_\b=\inf_{\cN_\b} \mathcal{R}_\b\geq \inf_{\HoneK\setminus\{0\}} \mathcal{R}_\b.
	\end{equation}
On the other hand, for any $\mf u\in \HoneK\setminus\{0\}$ there exists a unique $t(\mf u)>0$, given by
	\begin{equation*}
		t(\mf u)=\left(\frac{\sum_{i=1}^K \norm{u_i}_i^2}{\sum_{i=1}^K |u_i|_{Kq,i}^{Kq}+K\b|\prod_{i=1}^K u_i|_q^q}\right)^\frac{1}{Kq-2},
	\end{equation*}
	such that $t(\mf u)\mf u\in \cN_{\b}$. By homogeneity of $\mathcal{R}_\b$, for every $\mf{u} \in \HoneK\setminus\{0\}$
	\begin{equation*}
		\mathcal{R}_\b(\mf u)=\mathcal{R}_\b (t(\mf u)\mf u) \geq \inf_{\cN_\b} \mathcal{R}_\b = \ell_\b,
	\end{equation*}
	and hence
	\begin{equation}\label{st2062}
		\inf_{\HoneK\setminus\{0\}} \mathcal{R}_\b \ge \ell_\b.
	\end{equation}
	By \eqref{st2061}, we infer that in fact equality holds in \eqref{st2062}. At this point, since for any fixed $\mf u\in\HoneK\setminus\{0\}$ we have that $\b\mapsto \mathcal{R}_\b(\mf u)$ is non-increasing, we conclude that the same holds for $\ell_\b$.
	
	Let now $(\b_n)_n\subset (0,+\infty)$, $(\mf u_n)_n\subset H^1_{\mathrm{rad}}(\R^d,\R^K)$ be sequences such that $\b_n \searrow \bar\b$ as $n\to \infty$ and $\mf u_n$ is a fully non-trivial ground state for $I_{\beta_n}$, positive and radial, whose existence is ensured by \Cref{thm:beta grande}. By the monotonicity of $\b\mapsto \ell_\b$, $(\mf u_n)_n$ is bounded. Then, as $n\to\infty$,
	\begin{align*}
		|u_{i,n}|_{Kq,i}^2&\leq \frac{\norm{u_{i,n}}_i^2}{\bar S}=\frac{|u_{i,n}|_{Kq,i}^{Kq}+\b_n\left|\prod\nolimits_j u_{j,n}\right|_q^q}{\bar S} \\
		&\leq \frac{1+\bar\b+o(1)}{\bar S}|u_{i,n}|_{Kq,i}^{q}\left(|u_{i,n}|_{Kq,i}^{(K-1)q}+\left(\prod_{j=1}^K \mu_j\right)^{-\frac{1}{K}}\prod_{j \neq i}|u_{j,n}|_{Kq,j}^{q}\right) \\
		&\leq \frac{C(1+\bar \b+o(1))}{\bar S}|u_{i,n}|_{Kq,i}^{q}
	\end{align*}
	for $i=1,\dots,K$, where $C$ is independent of $n$. By the hypothesis on $q$, $\exists \delta>0$ independent of $n$ such that
	\begin{equation*}
		|u_{i,n}|_{Kq,i}\geq \delta \quad \forall n\in \N, \, \forall i=1,\dots,K.
	\end{equation*}
	It follows that, up to subsequences, there exists $\mf u\in H^1_{\mathrm{rad}}(\R^d,\R^K)$ fully non-trivial such that
	\begin{equation*}
		\mf u_n\to \mf u \text{ in } L^{Kq}(\R^d,\R^K), \quad \mf u_n\rightharpoonup \mf u \text{ in } H^1_{\mathrm{rad}}(\R^d,\R^K), \quad \mf u_n \to \mf u \text{ a.e. in } \R^d \quad \text{ as } n\to\infty.
	\end{equation*}
	Then, the weak formulation of \eqref{eq:general problem} for $\mf u_n$ with $\b=\b_n$ passes to the limit and we obtain that $\mf u$ is a weak solution of \eqref{eq:general problem} with $\b=\bar \b$. In particular, $\mf u\in \cN_{\b}$ and
	\begin{equation*}
		\norm{u_i}_i^2=|u_i|_{Kq,i}^{Kq}+\b|\prod_{j=1}^K u_j|_q^q=\lim_{n\to\infty} \left(|u_{i,n}|_{Kq,i}^{Kq}+\b_n|\prod_{j=1}^K u_{j,n}|_q^q\right)=\lim_{n\to\infty}\norm{u_{i,n}}_i^2.
	\end{equation*}
	Then $\mf u_n\to \mf u$ in $H^1_{\mathrm{rad}}(\R^d,\R^K)$ strongly, and hence
	\begin{equation*}
		I_\b(\mf u)=\lim_{n\to\infty}I_{\b_n}(\mf u_n)=\lim_{n\to\infty} \ell_{\b_n}\leq \ell_{\bar\b},
	\end{equation*}
	by monotonicity of $\ell_\b$, which concludes the proof. \qedhere
\end{proof}

\section{Least energy positive solutions in the weak cooperative and competitive case}\label{sec:esistenza minimo radiale beta piccolo e negativo}
In the first part of the section we prove Theorems \ref{thm:beta positivo} and \ref{thm:esistenza radiale beta negativo} through a careful analysis of minimizing sequences for $I_\b$ on $\cM_\b^r$, in the spirit of \cite[Section 2]{SoaveTavares2016}. We carry out the two proofs in a unified way as they share the same steps, through a series of lemmas. Afterwards, in Subsection \ref{sub: non-ex}, we prove Theorem \ref{thm:non esistenza non radiale beta negativo}.

The first lemma gives a uniform upper bound on $k_\b^r$.

\begin{lemma}
	There exists $\bar C>0$ such that 
	\begin{equation}\label{eq:controllo dall'alto k_b^r}
		k_\b^r\leq \bar{C} \qquad \text{for all $\beta \in \R$}.
	\end{equation}
	More precisely, one can take
	\begin{equation}\label{eq:definizione Cbar}
		\bar{C}=\frac{Kq-2}{2Kq}\max_i\frac{\left(1+\lambda_i^2\right)^\frac{Kq}{Kq-2}}{\mu_i^\frac{2}{Kq-2}}\inf \left\{\sum_{i=1}^K \tilde{S}(\Omega_i)^{\frac{Kq}{Kq-2}} \left| \begin{array}{l} \Omega_1,\dots,\Omega_K\subset \R^d \\\text{are open and radially symmetric, } \\ \text{and } \ \bigcap_{i=1}^K \Omega_i=\emptyset \end{array}\right.\right\}
	\end{equation}
	with $\tilde{S}(\Omega)$ the best Sobolev constant for the embedding $H^1_{\mathrm{rad}}(\Omega)\hookrightarrow L^{Kq}(\Omega)$, namely
	\begin{equation*}
		\tilde{S}(\Omega)=\inf_{u\in H^1(\Omega)\setminus \{0\}} \frac{\int_{\Omega} |\nabla u|^2+|u|^2}{\left(\int_{\Omega}|u|^{Kq}\right)^\frac{2}{Kq}}
	\end{equation*}
\end{lemma}
\begin{proof}
	Let $\mf u\in \cM_\b^r$ be such that $\prod_{i=1}^K u_i\equiv 0$. Then
	\begin{equation*}
		\norm{u_i}_i^2=\mu_i |u_i|_{Kq}^{Kq}.
	\end{equation*}
	As a consequence, for all $\b\in \R$
	\begin{equation*}
		k_\b^r\leq I_\b (\mf u)=\frac{Kq-2}{2Kq}\sum_{i=1}^{K}\left(\frac{\norm{u_i}_i^2}{|u_i|_{Kq,i}^{2}}\right)^\frac{Kq}{Kq-2}\leq \frac{Kq-2}{2Kq}\max_i \frac{(1+\lambda_i^2)^{\frac{Kq}{Kq-2}}}{\mu_i^{\frac{2}{Kq-2}}} \sum_{i=1}^K \left(\frac{\norm{u_i}_{H^1(\Omega_i)}^2}{|u_i|_{Kq}^2}\right)^\frac{Kq}{Kq-2}
	\end{equation*}
	Taking the infimum over all such $\mf u$ and noticing that, for every $\mf u\in H^1_{\mathrm{rad}}(\R^d,\R^K)$ such that $\prod_{i=1}^K u_i \equiv 0$, up to a suitable component-wise multiplication we have that $\mf u\in \cM_\b^r$, we conclude.
\end{proof}
\begin{remark}
	A priori the value of the infimum in the optimal partition problem \eqref{eq:definizione Cbar}
	could be zero. However, by \eqref{eq:controllo dall'alto k_b^r} this cannot occur. Indeed, in that case
	\begin{equation*}
		k_\b^r=0 \quad \forall \b\in \R,
	\end{equation*} 
	contradicting the fact that (see \Cref{rmk:equality of infima})
	\begin{equation*}
		\ell_\b=k_\b^r>0 \quad \text{for } \b>\bar\b.
	\end{equation*}
\end{remark}

\medskip

Before proceeding, we introduce a notation which will be convenient throughout the section. Define, for $i=1,\dots,K$,
\begin{equation}\label{eq:G_beta,i}
	G_{\b,i}(\mf u)=\norm{u_i}_i^2-\mu_i |u_i|_{Kq}^{Kq}-\b \left|\prod\nolimits_{j=1}^K u_j\right|_q^q,
\end{equation}
so that
\begin{equation*}
	\cM_\b^r=\left\{\mf u\in H^1_{\mathrm{rad}}(\R^d,\R^K) : \ u_i\neq 0, \ G_{\b,i}(\mf u)=0 \quad \forall i=1,\dots,K\right\}.
\end{equation*}
For $\mf u\in \cM_\b^r$ we compute
\begin{equation*}
	G_{\b,i}'(\mf u)[\bs \varphi]=2\langle u_{i},\varphi_i \rangle_i - Kq \mu_i\int_{\R^d} |u_{i}|^{Kq-2}u_{i}\varphi_i - q\b\sum_{j=1}^K \int_{\R^d} |u_{j}|^{q-2}u_{j}\varphi_j\prod_{k\neq j} |u_k|^q,
\end{equation*}
for every $\bs \varphi\in H^1_{\mathrm{rad}}(\R^d,\R^K)$. We introduce $\bs \varphi_{j}(\mf u)$ defined by
\begin{equation}\label{eq:varphi_j}
	(\bs \varphi_{j}(\mf u))_k=\begin{cases}
		u_{j} & k=j \\
		0 & k\neq j
	\end{cases}, \quad j,k=1,\dots,K,
\end{equation}
and we denote by $\bs M_{\b}(\mf u)\in \R^{K \times K}$ the matrix
\begin{equation}\label{eq:definition matrix M}
	(\bs M_\b(\mf u))_{ij} := (G'_{\b,i}(\mf u)[\bs \varphi_j(\mf u)])=\begin{cases}
		(2-Kq)\mu_i \left|u_{i}\right|_{Kq}^{Kq} + (2-q) \b\left|\prod\nolimits_k u_{k}\right|_q^q & i = j \\
		-q\b\left|\prod\nolimits_k u_{k}\right|_q^q & i \neq j
	\end{cases},
\end{equation}
where in the last equality we used that $G_{\b,i}(\mf u)=0$ for all $i=1,\dots,K$.

\medskip 

The next lemma tells that, under suitable assumptions on $q$, $\b$ and the energy of $\mf u$,  $\bs M_\b(\mf u)$ is negative definite.
\begin{lemma}\label{lemma: M_b(u) definita positiva}
	Let $\mf u\in \cM_\b^r$ and assume that either
	\begin{enumerate}
		\item $\b<0$, or \label{case 1}
		\item $0<\b<\ubar\beta$, $q\geq 2$ and $I_\b(\mf u)\leq 2\bar C$, \label{case 2}
	\end{enumerate}
	where 
	\begin{equation}\label{eq:definizione ubar beta}
		\ubar\b=\frac{\bar S(Kq-2)}{2q(K-1)}\left(\prod_{i=1}^K \mu_i^{\frac{1}{K}}\right) \left( \frac{4Kq}{Kq-2} \frac{\bar{C}}{\bar S} \right)^{\frac{2-Kq}{2}}
	\end{equation}
	and $\bar C$ is defined in \eqref{eq:definizione Cbar}.
	Then, $\bs M_\b(\mf u)$ is strictly negative definite.
\end{lemma}
\begin{proof}
	We prove that $-\bs M_\b(\mf u)$ is strictly diagonally dominant with positive diagonal entries: the thesis then follows from Gershgorin's circle theorem. We split the proof into the two cases.
	
	\textit{Case \eqref{case 1}:} 
	Since $\mf u\in \cM_\b^r$,
	\begin{equation}\label{eq:dis}
		\mu_i|u_i|_{Kq}^{Kq} = \norm{u_i}_i^2 -\beta |\prod\nolimits_{j} u_j|_q^q > |\b||\prod\nolimits_{j} u_j|_q^q,
	\end{equation}
	whence
	\begin{equation*}
	\begin{split}
		(Kq-2)\mu_i \left|u_{i}\right|_{Kq}^{Kq} + (q-2)\b |\prod\nolimits_j u_{j}|_q^q &=(Kq-2)\mu_i \left|u_{i}\right|_{Kq}^{Kq} - (q-2)|\b| |\prod\nolimits_j u_{j}|_q^q \\
		& >(K-1)q|\b||\prod\nolimits_j u_{j}|_q^q.
		\end{split}
	\end{equation*}
	Since the right hand side is non-negative, we deduce exactly that $-\bs M_\b(\mf u)$ is strictly diagonally dominant with positive diagonal entries.
%
	
	\textit{Case \eqref{case 2}:} We have to prove that, for every \( i \),
	\begin{equation}\label{eq:claim}
		(Kq-2)\mu_i \left|u_{i}\right|_{Kq}^{Kq} + (q-2)\b |\prod\nolimits_k u_{k}|_q^q > (K-1)q\b |\prod\nolimits_k u_{k}|_q^q.
	\end{equation}
	If $\mf u\in \cM_\b^r$ and $I_\b(\mf{u}) \le 2 \bar C$, then
	\[
	\left|u_{i}\right|_{Kq,i}^q \leq \frac{\norm{u_i}_i^q}{\bar S^{\frac{q}{2}}} \leq \frac{\left( \sum_{j=1}^K \norm{u_j}_j^2 \right)^{\frac{q}{2}}}{\bar S^{\frac{q}{2}}} \leq \left( \frac{2Kq}{Kq-2} \frac{2\bar{C}}{\bar S} \right)^{\frac{q}{2}} \quad \forall i.
	\]
	Hence, by the Hölder inequality,
	\begin{align*}
		\b|\prod\nolimits_j u_{j}|_q^q & \leq \b \left(\prod_{i=1}^K \mu_i^{-\frac{1}{K}}\right) \left|u_{i}\right|_{Kq,i}^{q} \prod\nolimits_{j\neq i} \left|u_{j}\right|_{Kq,j}^q \\
		& \leq \b \left(\prod_{i=1}^K \mu_i^{-\frac{1}{K}}\right) \left|u_{i}\right|_{Kq,i}^{q-2} \frac{\norm{u_i}_i^2}{\bar S} \prod\nolimits_{j\neq i} \left|u_j\right|_{Kq,j}^{q} \\
		& \leq \frac{\b}{\bar S} \left(\prod_{i=1}^K \mu_i^{-\frac{1}{K}}\right) \left( \frac{2Kq}{Kq-2} \frac{2\bar{C}}{\bar S} \right)^{\frac{Kq-2}{2}}\norm{u_i}_i^2.
	\end{align*}
	By definition of \( \ubar \b \) and $\b<\ubar \b$, this yields
	\[
	\b |\prod\nolimits_j u_j|_{q}^{q} < \frac{\norm{u_i}_i^2}{\frac{(K-2)q+2}{Kq-2}+1} = \frac{\mu_i\left|u_i\right|_{Kq}^{Kq} + \b|\prod\nolimits_j u_j|_{q}^q}{\frac{(K-2)q+2}{Kq-2}+1},
	\]
	which is equivalent to
	\[
	(Kq-2)\mu_i\left|u_i\right|_{Kq}^{Kq} > ((K-2)q+2)\b|\prod\nolimits_j u_j|_{q}^q,
	\]
	namely \eqref{eq:claim}.
\end{proof}

An important consequence is the following:
\begin{lemma}\label{lemma:cM_b is a manifold}
	If $\b<\ubar{\b}$, where $\ubar \b$ is defined in \eqref{eq:definizione ubar beta},
	\begin{equation*}
		\cH=\cM_\b^r\cap \left\{\mf u:I_\b(\mf u) < 2\bar{C}\right\}
	\end{equation*}
	is a smooth manifold of codimension $K$.
\end{lemma}

\begin{proof}
	Since 
	\begin{equation*}
		\cH=\left\{\mf u\in H^1_{\mathrm{rad}}(\R^d,\R^K) \setminus\{0\}: \ I_\b(\mf u)-2\bar C < 0, \ G_{\b,i}(\mf u)=0 \quad \forall i=1,\dots,K\right\},
	\end{equation*}
	where $G_{\b,i}$ is defined in \eqref{eq:G_beta,i}, it is enough to check that, for any $\mf u\in \cH$, the differential
	\begin{equation*}
		\Phi(\mf u)=\left(G_{\b,1}'(\mf u),\dots,G_{\b,K}'(\mf u)\right): H^1_{\mathrm{rad}}(\R^d,\R^K) \to \R^K 
	\end{equation*}
	is surjective. We use again the test functions $\bs \varphi_{j}(\mf u)\in H^1_{\mathrm{rad}}(\R^d,\R^K)$, defined in \eqref{eq:varphi_j}, and we show that
	\begin{equation*}
		\Phi(\mf u)[\bs\varphi_1(\mf u)], \, \dots, \,\Phi(\mf u)[\bs \varphi_K(\mf u)]
	\end{equation*}
	are linearly independent in $\R^K$. Indeed, let $t_1,\dots,t_K$ be real numbers such that
	\begin{equation*}
		\sum_{i=1}^K t_i \Phi(\mf u)[\bs \varphi_i(\mf u)]=0.
	\end{equation*}
	By \eqref{eq:definition matrix M}, this is equivalent to
	\begin{equation*}
		\sum_{i=1}^K t_i \left(M_{\b}(\mf u)\right)_{ji}=0 \quad \forall j=1,\dots,K.
	\end{equation*}
	By \Cref{lemma: M_b(u) definita positiva}, $M_{\b}(\mf u)$ is non-singular and hence $t_i=0$ for $i=1,\dots,K$.
\end{proof}

The next lemma tells that, under suitable assumptions on $q$ and $\b$, minimizing sequences cannot have vanishing components.

\begin{lemma}\label{lemma: succ minimizzanti controllate dal basso}
	Assume that one of the following holds:
	\begin{enumerate}[label=\alph*)]
		\item $q>2$ and $\b>0$, or \label{case a}
		\item $q=2$ and $0<\b< L$, where 
		\begin{equation}\label{eq:L}
			L= \left(\prod_{i=1}^K \mu_i^{\frac{1}{K}}\right) \frac{(K-1)^{K-1} S^K}{2^{K-1} K^{K-1}\bar{C}^{K-1}}, \quad \text{ or}
		\end{equation} \label{case b}
		\item $\b<0$. \label{case c}
	\end{enumerate} 
	Then, there exists $\delta=\delta(\bs{\lambda}, \bs{\mu}, \b, q,d)>0$ such that, for any sequence $(\mf u_n)\subset \cM_\b^r$ such that
	\begin{equation*}
		I_\b (\mf u_n)\to k_\b^r \text{ as } n\to\infty,
	\end{equation*}
	for some $n_0\in \N$ we have that
	\begin{equation*}
		|u_{i,n}|_{Kq,i}\geq \delta \quad \forall i=1,\dots,K, \quad \forall n\geq n_0. 
	\end{equation*}
\end{lemma}

\begin{proof}
	We divide the proof into the three cases.
	
	\textit{Case \ref{case a}:} since $I_\b(\mf u_n)\to k_\b^r$ and by Sobolev embedding, for $i=1,\dots,K$ and $n\geq n_0$ for some $n_0$ we have
	\begin{equation}\label{eq:controllo dall'altro norma Kq componenti}
		|u_{i,n}|_{Kq,i}^2\leq\frac{\norm{u_{i,n}}^2_i}{\bar S}\leq\frac{\sum_{j=1}^K\norm{u_{j,n}}_j^2}{\bar S}\leq\frac{2Kq}{Kq-2}\frac{2k_\b^r}{\bar S}\leq\frac{2Kq}{Kq-2}\frac{2\bar{C}}{\bar S}.
	\end{equation}
	Then
	\begin{align*}
		|u_{i,n}|_{Kq,i}^2 &\leq \frac{\|u_{i,n}\|_i^2}{\bar S} = \frac{|u_{i,n}|_{Kq,i}^{Kq} + \beta |\prod_{j=1}^K u_{j,n}|_q^q}{\bar S} \\
		&\leq \frac{|u_{i,n}|_{Kq,i}^{Kq}}{\bar S} + \frac{\beta}{\bar S} \left(\prod_{j=1}^K \mu_j^{-\frac{1}{K}}\right) |u_{i,n}|_{Kq,i}^q \prod_{j \neq i} |u_{j,n}|_{Kq,j}^q \\
		&\overset{\eqref{eq:controllo dall'altro norma Kq componenti}}{\leq}
		\frac{|u_{i,n}|_{Kq,i}^{Kq}}{\bar S} + \frac{\beta}{\bar S} \left(\prod_{j=1}^K \mu_j^{-\frac{1}{K}}\right) \left( \frac{2Kq}{Kq - 2} \cdot \frac{2\bar{C}}{\bar S} \right)^\frac{(K-1)q}{2} |u_{i,n}|_{Kq}^q.
	\end{align*}
	This yields
	\begin{equation*}
		1 \leq \frac{\mu_i}{\bar S} |u_{i,n}|_{Kq,i}^{Kq - 2} + \frac{\beta}{\bar S} \left(\prod_{j=1}^K \mu_j^{-\frac{1}{K}}\right) \left( \frac{2Kq}{Kq - 2} \cdot \frac{2\bar{C}}{\bar S} \right)^\frac{(K-1)q}{2} |u_{i,n}|_{Kq,i}^{q - 2},
	\end{equation*}
	which clearly implies the thesis for some $\delta>0$.
	
	\textit{Case \ref{case b}}: arguing similarly to \eqref{eq:controllo dall'altro norma Kq componenti}, for any $\eps>0$ there exists $n_0\in\N$ such that for $n\geq n_0$ and $i=1,\dots,K$
	\begin{equation}\label{eq:controllo dall'altro norma 2K componenti}
		|u_{i,n}|_{2K,i}^2\leq\frac{\norm{u_{i,n}}^2_i}{\bar S}\leq\frac{\sum_{j=1}^K\norm{u_{j,n}}_j^2}{\bar S}\leq \frac{2K}{K-1}\frac{k_\b^r(1+\eps)}{\bar S}\leq\frac{2K(1+\eps)\bar{C}}{(K-1)\bar S}.
	\end{equation}
	Then
	\begin{align*}
		|u_{i,n}|_{2K,i}^2 &\leq \frac{\|u_{i,n}\|_i^2}{\bar S} = \frac{|u_{i,n}|_{2K,i}^{2K} + \beta |\prod_{j=1}^K u_{j,n}|_2^2}{\bar S} \\
		&\leq \frac{|u_{i,n}|_{2K,i}^{2K}}{\bar S}  + \frac{\beta}{\bar S} \left(\prod_{j=1}^K \mu_j^{-\frac{1}{K}}\right) |u_{i,n}|_{2K,i}^2 \prod_{j \neq i} |u_{j,n}|_{2K,j}^2 \\
		&\overset{\eqref{eq:controllo dall'altro norma 2K componenti}}{\leq}
		\frac{|u_{i,n}|_{2K,i}^{2K}}{\bar S}  + \frac{\beta}{\bar S} \left(\prod_{j=1}^K \mu_j^{-\frac{1}{K}}\right) \left( \frac{2K(1+\eps)\bar{C}}{(K-1)\bar S} \right)^{K-1} |u_{i,n}|_{2K,i}^2.
	\end{align*}
	
	Thus, letting $\lambda=\frac{\b}{L}<1$,
	\begin{equation*}
		1\leq \frac{|u_{i,n}|_{2K,i}^{2K-2}}{\bar S}+\b\frac{2^{K-1}K^{K-1}(1+\eps)^{K-1}\bar{C}^{K-1}}{(K-1)^{K-1}\bar S^K}=\frac{|u_{i,n}|_{2K,i}^{2K-2}}{\bar S}+\lambda(1+\eps)^{K-1}.
	\end{equation*}
	Choosing $\eps>0$ small enough such that
	\begin{equation*}
		\lambda(1+\eps)^{K-1}=\frac{1+\lambda}{2},
	\end{equation*}
	this implies $|u_{i,n}|_{2K,i}\geq \delta$ for all $i$ and for $n\geq n_0$, for some $\delta>0$ independent of $n$ and $i$.
	
	\textit{Case \ref{case c}}: since $\mf u_n\in \cM_\b^r$,
	\begin{equation*}
		|u_{i,n}|_{Kq,i}^2 \leq \frac{\|u_{i,n}\|_i^2}{\bar S} = \frac{|u_{i,n}|_{Kq,i}^{Kq} + \beta |\prod_{j=1}^K u_{j,n}|_q^q}{\bar S}\leq \frac{|u_{i,n}|_{Kq,i}^{Kq}}{\bar S}
	\end{equation*}
	which clearly implies the thesis since $Kq>2$.
\end{proof}

Thanks to the last lemma, as long as we are able to extract a converging subsequence from a minimizing sequence, its limit will be fully non-trivial. To get a converging subsequence, we need to recover some compactness, which is represented by the Palais-Smale condition, whose validity is investigated in the following lemma. For the Palais-Smale condition for problems with a constraint, we refer to \cite[Chapter 4]{Kav} (see also Exercise 11 therein, which concerns precisely the situation of a constraint with higher codimension).

\begin{lemma}\label{lemma:PS holds}
	If $\b<\ubar \b$, $(PS)_{k_\b^r}$ for $I_\b|_{\cM_\b^r}$ holds.
\end{lemma}
\begin{proof}
We wish to show that, for every sequence $\{(\mf u_n, \gamma_{n,1},\dots, \gamma_{n,K})\}\subset \cM_\b^r \times \R^K$ such that
	\begin{equation}\label{eq:palais smale condition on vector nehari}
		I_\b(\mf u_n)\to k_\b^r, \quad I_\b'(\mf u_n)=\sum_{i=1}^K \gamma_{n,i} G_{\b,i}'(\mf u_n)+o(1) \quad \text{as }n\to\infty,
	\end{equation}
there exists a converging subsequence in 	$\cM_\b^r \times \R^K$. Testing the second equation in \eqref{eq:palais smale condition on vector nehari} with $\bs \varphi_{i}(\mf u_n)$, where $\bs \varphi_i(\mf u_n)$ is defined as in \eqref{eq:varphi_j}, namely
	\begin{equation*}
		(\bs \varphi_{i}(\mf{u}_n))_j=\begin{cases}
			u_{i,n} & j=i \\
			0 & j\neq i
		\end{cases}, \quad \text{for } n\in \N,\, i=1,\dots,K,
	\end{equation*}
	we find
	\begin{equation*}
		\sum_{j=1}^K(\bs M_\b(\mf u_n))_{ji}\,\gamma_{n,j}=o(1) \quad \forall i=1,\dots,K.
	\end{equation*}
	as $n\to \infty.$
	Thanks to \eqref{eq:Ibeta on nehari} and the first equation in \eqref{eq:palais smale condition on vector nehari}, $(\mf u_n)_n$ is bounded in $\left(H^1_{\mathrm {rad}}(\R^d)\right)^K$.
	Then, up to subsequences, $\mf u_n \rightharpoonup \mf u$ weakly in $H^1_{\mathrm {rad}}(\R^d,\R^K)$ for some limit function $\mf u$, and $\mf u_{n}\to \mf u$ strongly in $L^{Kq}(\R^d, \R^K)$. Using the fact that $\mf u_n \in \cM_\b^r$ and that eventually $I_\b(\mf u_n)\leq 2\bar C$, by \Cref{lemma: M_b(u) definita positiva} we obtain that $\gamma_{n,i}\to 0$ as $n\to\infty$ for $i=1,\dots,K.$ Hence, recalling \eqref{eq:palais smale condition on vector nehari}, we obtain that
	\begin{equation}\label{eq:I'beta infinitesimal}
		I_\b' (\mf u_n)\to 0 \quad \text{as } n\to\infty.
	\end{equation} 
	As a consequence, $(\mf u_n)_n$ is a Palais-Smale sequence for the unconstrained functional $I_\b$. It is now standard to prove that $\mf u_n\to \mf u$ in $H^1_{\mathrm {rad}}(\R^d,\R^K)$. By weak convergence and \eqref{eq:I'beta infinitesimal},
	\begin{equation*}
		I_\b'(\mf u_n)[\mf u_n-\mf u]\to 0 \quad \text{and} \quad I'_\b(\mf u)[\mf u_n - \mf u]\to 0.
	\end{equation*}
	Then,
	\begin{multline*}
		o(1)=\left(I_\b'(\mf u_n)-I'_\b(\mf u)\right)[\mf u_n-\mf u]
		\\
		=\sum_{i=1}^K\norm{u_{i,n}-u_i}_i^2
		-\sum_{i=1}^K\mu_i\int_{\R^d} \left(\oddpower{u_{i,n}}{Kq-2}-\oddpower{u_i}{Kq-2}\right)\left(u_{i,n}-u_i\right)\\
		-\b\sum_{i=1}^K\int_{\R^d}\left(\oddpower{u_{i,n}}{q-2}\prod_{j\neq i}|u_{j,n}|^q-\oddpower{u_i}{q-2}\prod_{j\neq i}|u_j|^q\right)\left(u_{i,n}-u_i\right).
	\end{multline*}
	Since, by $L^{Kq}$-convergences, the last two terms tend to $0$, this implies the desired convergence. Then the equations defining $\cM_\b^r$ pass to the limit and we obtain $\mf u\in \cM_\b^r$, concluding the proof.
\end{proof}

The results established so far allow us to prove the existence of a minimizer for $I_\b$ on $\cM_\b^r$, as we will argue later. Before proving \Cref{thm:beta positivo}, the next proposition proves that minimizers solve \eqref{eq:general problem}. Notice that this is not trivial, since $\cM_\b^r$ is not the standard Nehari manifold of codimension $1$, which is well-known to be a natural constraint, but a manifold of codimension $3$.

\begin{proposition}\label{prop:minimizers are solutions}
	Let $\mf u_{\mathrm{min}} \in \cM_\b^r$ be a minimizer for $I_\b$, that is,
	\begin{equation*}
		I_\b(\mf u_{\mathrm{min}})=k_\b^r.
	\end{equation*}
	If $\b<\ubar\b$, where $\ubar\b$ is defined in \eqref{eq:definizione ubar beta}, then $\mf u_{\mathrm{min}}$ is a (weak) solution of \eqref{eq:general problem}.
\end{proposition}
\begin{proof}	
	
	Since $\cM_\b^r\cap \left\{I_\b < 2\bar{C}\right\}$ is a smooth manifold of codimension $K$ by \Cref{lemma:cM_b is a manifold}, by the Lagrange multipliers rule there exist \( \gamma_1,\dots,\gamma_K \in \mathbb{R} \) such that 
	\[
	I_\b'(\mf u_{\mathrm{min}})-\sum_{i=1}^K\gamma_i G_{\b,i}'(\umin)=0.
	\]
We show that \( \gamma_i = 0 \) for \( i=1,\dots,K \). Testing with \( \bs \varphi_i(\mf u_{\mathrm{min}}) \), \( i=1,\dots,K \), defined as in \eqref{eq:varphi_j}, and using that \( \umin \in \cM_\b^r \), one finds that \( \bs \gamma = (\gamma_1, \dots, \gamma_K) \) solves \( \bs M_\b(\mf u_{\mathrm{min}}) \bs \gamma = \bs 0 \). Since \( \bs M_\b(\mf u_{\mathrm{min}}) \) is invertible by \Cref{lemma: M_b(u) definita positiva}, it follows that \( \gamma_i = 0 \) for \( i=1,\dots,K \).  
\end{proof}

\begin{proof}[Conclusion of the proof of \Cref{thm:beta positivo}]
	Notice that, evaluating $\ubar \b$ for $q=2$, one finds that $\ubar \b<L$ (recall that $\ubar{\b}$ and $L$ are defined in \eqref{eq:definizione ubar beta} and \eqref{eq:L}, respectively). Then, by hypothesis, either case \ref{case a} or \ref{case b} of \Cref{lemma: succ minimizzanti controllate dal basso} holds, and
	\begin{equation*}
		\cH=\cM_\b^r\cap \left\{I_\b \leq 2\bar{C}\right\}
	\end{equation*}
	is closed in $H^1_{\mathrm{rad}}(\R^d,\R^K)$. Indeed, the equations defining $\cM_\b^r$ clearly pass to the limit under strong convergence, and pointwise convergence preserves radial symmetry; moreover, \Cref{lemma: succ minimizzanti controllate dal basso} ensures that limits of sequences in $\cH$ cannot have a null component. By Ekeland's variational principle (and taking into account that any minimizing sequence in $\cH$ stays away from the relative boundary $\cM_\b^r \cap \{I_\b=2 \bar C\}$, since $k_\b^r \le \bar C<3/2 \bar C$) there exists a Palais-Smale sequence $(\mf u_n)_n$ for $I_\b|_{\cM_\b^r}$ at level $k_\b^r$. It is not restrictive to assume $\mf{u}_n\geq 0$ on $\R^d$ for all $n$, and \Cref{lemma:PS holds} ensures the existence of a non-negative minimizer $\umin$. At this point \Cref{prop:minimizers are solutions} implies that minimizers are solutions of \eqref{eq:general problem}, and the strict positivity follows from the strong maximum principle. 
	
It remains to show that the positive minimizer found by minimizing $I_\b|_{\cM_\b^r}$ is a least energy positive solution (without the radial constraint). This is a simple consequence of the fact that, by \cite{Busca2000}, any positive solution is radially symmetric, and hence stays in $\cM_\b^r$.
\end{proof}

\begin{proof}[Conclusion of the proof of \Cref{thm:esistenza radiale beta negativo}]
	Under the assumption of the theorem, \ref{case c} of \Cref{lemma: succ minimizzanti controllate dal basso} holds. Then, arguing exactly as above, we complete the proof.
\end{proof}

\subsection{Non-existence of least energy solutions in the competitive case}\label{sec:non existence non radial b<0}
The aim of this subsection is proving the non-existence result stated in \Cref{thm:non esistenza non radiale beta negativo}.

\begin{proof}[Proof of \Cref{thm:non esistenza non radiale beta negativo}]\label{sub: non-ex}
	We follow the argument used in \cite[Lemma 3]{Mandel2014cooperative}. Since $\b<0$, notice that for any $\mf u\in \cM_\b$
	\begin{equation*}
		\norm{u_i}_i^2=\mu_i|u_i|_{Kq}^{Kq}+\b |\prod_{j=1}^K u_j|_q^q\leq \mu_i |u_i|_{Kq}^{Kq}, \quad i=1,\dots,K.
	\end{equation*}
	Then
	\begin{align*}
		I_\b(\mf u)&=\frac{Kq-2}{2Kq}\sum_{i=1}^{K}\norm{u_i}_i^2 \geq \frac{Kq-2}{2Kq}\sum_{i=1}^K \norm{u_i}_i^2 \left(\frac{\norm{u_i}_i^2}{\mu_i|u_i|_{Kq}^{Kq}}\right)^\frac{2}{Kq-2} \\
		&=\frac{Kq-2}{2Kq}\sum_{i=1}^K\left(\frac{\norm{u_i}_i^{2}}{ |u_i|_{Kq,i}^{2}}\right)^{\frac{Kq}{Kq-2}} \geq \frac{Kq-2}{2Kq}\sum_{i=1}^K c_{Kq,\lambda_i,\mu_i}^{\frac{Kq}{Kq-2}} \\
		&=\sum_{i=1}^K I_\b\left(\mf w_i\right),
	\end{align*}
	where $c_{Kq,\lambda_i,\mu_i}$ is defined in \eqref{c_p} and $\mf w_i$ in \eqref{eq:w_i definition}.
	Moreover, notice that equality can occur only if it occurs in the two previous inequalities: in the case of equality, the first implies $|\prod_{i=1}^K u_i|_q^q=0$, while the second requires $u_i$ to be a multiple and a translation $w_i$ for each $i=1,\dots,K.$ The two conditions are clearly incompatible, hence
	\begin{equation*}
		I_\b(\mf u)>\sum_{i=1}^K  I_\b (\mf w_i) \quad \forall \mf u\in \cM_\b.
	\end{equation*}
	Passing to the infimum,
	\begin{equation*}
		\ell_\b=\inf_{\cM_\b} I_\b \geq \sum_{i=1}^K  I_\b (\mf w_i).
	\end{equation*}
	
	We now show the other inequality, concluding the proof. Let $\eta\in C_c^\infty(\R^n)$ be a cut-off function such that 
	\begin{equation*}
		\eta=\begin{cases}
			0 & |x|\geq 1 \\
			1 & |x|\leq \frac{1}{2},
		\end{cases}
	\end{equation*}
	and $\left\{e_i\right\}$ be the standard basis of $\R^d$.
	Define for $R>0$
	\begin{equation*}
		\tilde u_{R,i}(x)=\eta\left(\frac{x}{R}\right)w_i(x),
	\end{equation*}
	where $w_i=w_{Kq,\lambda_i,\mu_i}$ is the unique positive solution to \eqref{eq:schrodinger 1 componente}, and
	\begin{align*}
		u_{R,1}(x)&=\tilde{u}_{R,1}(x-2Re_1)\left(\frac{\norm{\tilde u_{R,1}}_1^2}{|\tilde u_{R,1}|_{Kq,1}^{Kq}}\right)^\frac{1}{Kq-2}, \\
		\qquad u_{R,i}(x)&=\tilde{u}_{R,i}(x)\left(\frac{\norm{\tilde u_{R,i}}_i^2}{|\tilde u_{R,i}|_{Kq,i}^{Kq}}\right)^\frac{1}{Kq-2}, \quad \forall i=2,\dots, K.
	\end{align*}
	In such a way clearly
	\begin{align*}
		\supp (u_{R,1})&\subset \overline{B_R(2Re_1)}, \\
		\supp (u_{R,i})&\subset \overline{B_R(0)} \quad \forall i=2,\dots,K.
	\end{align*}
	so that $\prod_{i=1}^K u_{R,i}=0$, and by construction
	\begin{equation*}
		\norm{u_{R,i}}_i^2=\frac{\norm{\tilde u_{R,i}}_i^\frac{2Kq}{Kq-2}}{|\tilde u_{R,i}|_{Kq,i}^\frac{2Kq}{Kq-2}}=|u_{R,i}|_{Kq,i}^{Kq} \quad i=1,\dots,K.
	\end{equation*}
	Then $\mf u_R=\left(u_{R,1},\dots,u_{R,K}\right)\in \cM_\b$ for any $R>0$. As a consequence, 
	\begin{equation*}
		\inf_{\cM_\b} I_\b\leq \lim_{R\to\infty}I_\b(\mf u_R)=\frac{Kq-2}{2Kq}\lim_{R\to\infty} \sum_{i=1}^{K}\norm{u_{R,i}}_i^2=\frac{Kq-2}{2Kq}\lim_{R\to\infty} \sum_{i=1}^K \left(\frac{\norm{\tilde u_{R,i}}_i^2}{|\tilde u_{R,i}|_{Kq,i}^2}\right)^\frac{Kq}{Kq-2}.
	\end{equation*}
	Notice that, as $R\to +\infty$,
	\begin{equation*}
		\norm{\tilde u_{R,i}}_i\to \norm{w_i}_i, \quad |\tilde u_{R,i}|_{Kq,i}\to |w_i|_{Kq,i}. 
	\end{equation*}
	Then
	\begin{equation*}
		\inf_{\cM_\b} I_\b\leq \frac{Kq-2}{2Kq}\sum_{i=1}^K \left(\frac{\norm{\tilde w_i}_i^2}{|\tilde w_i|_{Kq,i}^2}\right)^\frac{Kq}{Kq-2}=\sum_{i=1}^K  I_\b (\mf w_i).
		\qedhere
	\end{equation*}
\end{proof}

\section{Asymptotic behaviour under strong competition}\label{sec:proof asymptotic}
This section is devoted to the proof of \Cref{thm:asintotica beta -infinito} and \Cref{thm: asy beh}.
\begin{proof}[Proof of \Cref{thm:asintotica beta -infinito}]
	We use the method of \emph{fibering map}, 
	used in a similar situation in \cite{Mandel2014repulsive}. 
	\medskip
		
	\textit{Step 1: equivalent formulation of the minimization problem.} 
		
	For each $\mf u\in \HoneradK$ such that $\prod_{i=1}^K u_i \equiv 0$, define the map
	\begin{equation*}
		h_{\mf u}(\bs t):=I_{-\infty}(t_1 u_1, \dots, t_K u_K) = \frac{1}{2} \sum_{i=1}^K t_i^2 \|u_i\|_i^2 - \frac{1}{Kq} \sum_{i=1}^K t_i^{Kq} |u_i|_{Kq,i}^{Kq}, \quad \bs t \in \R_+^K.
	\end{equation*}
	It is straightforward to see that $h_{\mf u}$ has a unique critical point $\bs t(\mf u)$, corresponding to the global maximum, given by 
	\begin{equation*}
			t_i(\mf u)=\left(\frac{\norm{u_i}_i^2}{|u_i|_{Kq,i}^{Kq}}\right)^{\frac{1}{Kq-2}}, \quad k=1,\dots,K,
		\end{equation*}
	such that $\left(t_1(\mf u) u_1, \dots, t_K(\mf u) u_K\right)\in \cM_{-\infty}^r $.
	Then 
	\begin{equation}\label{var car inf}
	\begin{split}
		\inf_{\cM_{-\infty}^r} I_{-\infty}&=\inf \left\{\max_{\bs t \in \R_+^K} h_{\mf u}(\bs t):\mf u\in \HoneradK, \ \prod_{i=1}^K u_i \equiv 0 \right\} \\
		&= \inf \left\{\frac{Kq-2}{2Kq}\sum_{i=1}^K \left(\frac{\norm{u_i}_i}{|u_i|_{Kq,i}}\right)^{\frac{2Kq}{Kq-2}}: \mf u\in \HoneradK, \ \prod_{i=1}^K u_i \equiv 0\right\} \\
		&=\inf \left\{\frac{Kq-2}{2Kq} \bar{I}(\mf u): \mf u\in \HoneradK, \ \prod_{i=1}^K u_i \equiv 0 \right\},
	\end{split}
	\end{equation}
	where
	\begin{equation*}
		\bar{I}(\mf u):=\sum_{i=1}^K \left(\frac{\norm{u_i}_i}{|u_i|_{Kq,i}}\right)^{\frac{2Kq}{Kq-2}}.
	\end{equation*}
	Also, whenever the last infimum is achieved by some $\mf u\in \HoneradK$ such that $\prod_{i=1}^K u_i \equiv 0$, then $\left(t_1(\mf u) u_1, \dots, t_K(\mf u) u_K\right)\in \cM_{-\infty}^r$ achieves the first infimum.
	
	\medskip
	
	\textit{Step 2: existence of a minimizer.} 
	
	We now prove the existence of a minimizer for $\bar I$. Notice that
	\begin{equation}\label{eq:homog of deg 0 of Ibar}
			\bar I(\alpha_1 u_1, \dots, \alpha_K u_K)=\bar I(u_1, \dots, u_K) \quad \forall\, \alpha_i>0, \, i=1,\dots,K.
		\end{equation}
Then it is not restrictive to take a minimizing sequence $(\mf u_n)_n$ for $\bar I$ such that $u_{i,n}\geq 0$ and $|u_{i,n}|_{Kq,i}=1$ for all $i=1,\dots,K$ and $n$. Clearly $(\mf u_n)_n$ is also bounded in $\HoneradK$ and hence, up to subsequences, the sequence converges weakly in $\HoneradK$, strongly in $L^{Kq}(\R^d,\R^K)$ and pointwise a.e. to a non-trivial and non-negative radial function $\mf u$ such that $\prod_{i=1}^K u_i\equiv 0$. Moreover, since $\bar I$ is weakly lower semicontinuous, $\mf u$ minimizes $\bar I$, and hence a suitable multiple, say $\umin$, minimizes $I_{-\infty}$ on $\cM_{-\infty}^r$. 

\medskip

\textit{Step 3: a minimizer of $k_{-\infty}^r$ solves \eqref{eq:eq schrodinger limite all'infinito}} 

Observe that we can take $\umin$ to be continuous on $\R^d\setminus \left\{0\right\}$, so that
$\left\{x\in \R^d:\uminarg{i}(x)\neq 0\right\}$ is open in $\R^d$ for $i=1,\dots,K.$ By minimality of $\umin$ for $\bar I$ with respect to variations $\bs \varphi\in \HoneradK$ such that $\supp(\varphi_i)\subset \{\uminarg{i}\neq 0\}$ for $i=1,\dots,K$, we find that
	\begin{equation*}
		0=\left. \frac{d}{dt} \right|_{t=0} \bar I(\umin+t\bs \varphi)=\sum_{i=1}^{K}\left(\langle \uminarg{i},\varphi_i\rangle_i-\mu_i \int_{\R^d} \uminarg{i}^{Kq-1}\varphi_i\right).
	\end{equation*}
	By arbitrariness of $\bs \varphi$ and elliptic regularity, $\umin$ solves \eqref{eq:eq schrodinger limite all'infinito}.
	\medskip
	
	\textit{Step 4: convergence as $\beta \to -\infty$}
	
	Now, let $(\b_n)_{n\in \N}$ be a sequence such that $\b_n\to-\infty$ as $n\to\infty$, and $(\mf u_n)_n$ be a sequence of corresponding minimizers for $I_{\b_n}|_{\cM_{\b_n}^r}$, whose existence is granted by \Cref{thm:esistenza radiale beta negativo}. Notice that $\cM_\b^r\supset\cM_{-\infty}^r$ for any $\b\in \R$ and $I_\b|_{\cM_{-\infty}^r}=I_{-\infty}$; then for all $n\in\N$
	\begin{equation}\label{ord infty}
		0<I_{\b_n}(\mf u_n)=\min_{\cM_{\b_n}^r} I_{\b_n}\leq \min_{\cM_{-\infty}^r} I_{-\infty}.
	\end{equation}
	It follows that $(\mf u_n)_n$ is bounded in $\HoneradK$ and, up to subsequences, $\mf u_n\to \mf u$ strongly in $L^{Kq}(\R^d,\R^K)$, $\mf u_n\rightharpoonup \mf u$ weakly in $\HoneradK$, and $\mf u_n \to \mf u$ a.e. in $\R^d$ as $n\to\infty$, for some $\mf u\in \HoneradK$.
	Moreover,
	\begin{equation*}
		\bar S|u_{i,n}|_{Kq,i}^2\leq \norm{u_{i,n}}_i^2=|u_{i,n}|_{Kq,i}^{Kq}+\b_n\left|\prod\nolimits_j u_{j,n}\right|_q^q \leq |u_{i,n}|_{Kq,i}^{Kq},
	\end{equation*}
	whence we deduce that $u_i\neq 0$ for $i=1,\dots,K$. Passing to the limit in the last inequality, by weak convergence in $H^1$,
	\begin{equation}\label{eq:inequality for infty}
		\norm{u_i}_i^2\leq |u_i|_{Kq,i}^{Kq}.
	\end{equation} 
	
	Now
	\begin{equation*}
		|\prod_{i=1}^K u_i|_q^q=\lim_{n\to\infty} |\prod_{i=1}^K u_{i,n}|_q^q=\lim_{n\to\infty} |\b_n|^{-1}\left(\norm{u_{1,n}}_1^2-|u_{1,n}|_{Kq,1}^2\right)=0,
	\end{equation*}
	since the norms $L^{Kq}$ and $H^1$ are uniformly bounded. Then
	\begin{equation*}
		\prod_{i=1}^K u_i\equiv 0 \quad \text{in $\R^d$}.
	\end{equation*}
	In addition, 
	\begin{align*}
		\frac{Kq-2}{2Kq}\sum_{i=1}^{K}\norm{u_i}_i^2&\leq \frac{Kq-2}{2Kq}\liminf_{n\to\infty} \sum_{i=1}^K \norm{u_{i,n}}_i^2=\liminf_{n\to\infty} \min_{\cM_{\b_n}^r}I_{\b_n} \\
		&\leq\limsup_{n\to\infty}\min_{\cM_{\b_n}^r}I_{\b_n} \overset{\eqref{ord infty}}{\le}\min_{\cM_{-\infty}^r}I_{-\infty} \overset{\eqref{var car inf}}{\leq} \frac{Kq-2}{2Kq} \bar I(\mf u) \\
		& =\frac{Kq-2}{2Kq} \sum_{i=1}^K \left(\frac{\norm{u_i}_i^2}{|u_i|_{Kq,i}^2}\right)^{\frac{Kq}{Kq-2}} \overset{\eqref{eq:inequality for infty}}{\leq} \frac{Kq-2}{2Kq}\sum_{i=1}^{K}\norm{u_i}_i^2.
	\end{align*}
	Then each inequality is in fact an equality, whence it follows that
	\begin{equation*}
		\lim_{\b\to-\infty} k_\b^r=k_{-\infty}^r,
	\end{equation*}
	and 
	\begin{equation*}
		\norm{u_i}_i^2=|u_i|_{Kq,i}^{Kq} \qquad \forall i=1,\dots,K;
	\end{equation*}
	that is, $\mf u \in \cM_{-\infty}^r$ and $I_{-\infty} (\mf u)=k_{-\infty}^r$. In particular, from what we proved before, we deduce that $\mf u$ solves \eqref{eq:eq schrodinger limite all'infinito}.
	Moreover, by weak $H^1-$convergence together with convergence of the norms, we have that	$\mf u_n \to \mf u$ in $\HoneradK$, and furthermore
	\begin{equation*}
		0\leq \limsup_{n\to\infty} |\b_n| \left|\prod\nolimits_{i=1}^K u_{i,n}\right|_q^q=\limsup_{n\to\infty}\left(|u_{1,n}|_{Kq,i}^{Kq}-\norm{u_{1,n}}_1^2\right)=0,
	\end{equation*}
	which completes the proof.
	\end{proof}
	
Now we address the characterization of the minimizers for $k_{-\infty}^r$, Theorem \ref{thm: asy beh}. It is convenient to state a preliminary result on the least energy level of the scalar problem
\begin{equation}\label{sc pb dom}
\begin{cases} 
-\Delta v+\lambda v=\mu v^{p-1}, \quad v>0 & \text{ in $\Omega$} \\
v=0 & \text{on $\pa \Omega$},
\end{cases}
\end{equation}
where $\Omega \subset \R^d$ is a radial domain, $\lambda, \mu>0$, and $p \in (2,2^*)$. A least energy radial positive solution to \eqref{sc pb dom} can be obtained as a constrained minimizer of
\[
E_{\Omega}(v):=\int_{\Omega} \frac12(|\nabla v|^2+ \lambda v^2)-\frac{\mu}{p}|v|^p
\]
on the Nehari manifold $\cN_{\Omega}:= \{u \in H_{0, \mathrm{rad}}^1(\Omega) \setminus \{0\}: \ E_{\Omega}'(v)[v]=0\}$. We denote by 
\[
c_\Omega:= \inf_{\cN_{\Omega}} E_{\Omega}
\]
It is well known that such a least energy positive solution exists, and that any constrained minimizer is a least energy positive solution, up to the a multiplication by $-1$.

\begin{lemma}\label{lem: monot inclusion}
If $\Omega_1 \subset \Omega_2$, then $c_{\Omega_1} > c_{\Omega_2}$.
\end{lemma}
\begin{proof}
It is plain that $c_{\Omega_1} \ge  c_{\Omega_2}$. Thus, we suppose by contradiction that $c_{\Omega_1} = c_{\Omega_2}$. Let $v_1$ be a least energy positive solution for $c_{\Omega_1}$. Since $\Omega_1 \subset \Omega_2$, by extending $v_1$ as $0$ outside of $\Omega_1$ (and denoting still by $v_1$ such an extension) we have that $v_1 \in \cN_{\Omega_2}$, and hence is also a minimizer for $c_{\Omega_2}$. But any minimizer is a least energy positive solution to \eqref{sc pb dom} in $\Omega_2$, and this contradicts the fact that $v_1=0$ in $\Omega_2 \setminus \Omega_1$.
\end{proof}
	
	\begin{proof}[Proof of \Cref{thm: asy beh}]
	Let $\mf{v}$ be a minimizer for $k_{-\infty}^r$. Throughout this proof, we will denote by $v_i$ both the function defined in $\R^d$, and its radial profile defined on $[0,+\infty)$. This should not be source of misunderstanding.
	
	\medskip
	
\textit{Step 1: each set $\{v_i\neq 0\}$ is connected, and hence is a ball, or an annulus, or the complement of a ball, or $\R^d$, and each component has constant sign} 

	Suppose by contradiction that $\{v_1\neq 0\}$ is disconnected, and let $\omega_1$ be one of its connected components. Clearly, 
	\[
	\int_{\{v_1\neq 0\} \setminus \omega_1} |\nabla v_1|^2+\lambda_1 v_1^2 >0,
	\]
	whence it follows that
	\[
	k_{-\infty}^r = I_{-\infty}(\mf{v}) = \sum_{i=1}^K \left(\frac12-\frac1{Kq}\right)\|v_i\|_i^2 > \sum_{i=2}^K \left(\frac12-\frac1{Kq}\right)\|v_i\|_i^2  + \int_{\omega_1} |\nabla v_1|^2+\lambda_1 v_1^2.
	\]
	The last term is precisely the energy $I_{-\infty}$ of the function $(v_1|_{\omega_1}, v_2,\dots,v_K)$, which belongs to $\cM_{-\infty}^r$ since $\mf{v}$ solves the limit problem \eqref{eq:eq schrodinger limite all'infinito}. Therefore, we find a contradiction with the minimality of $\mf{v}$.
	
	\medskip
	
\textit{Step 2: the set where two components vanish together has empty interior.}

From now on, by Step 1, up to changing sign in some components we may assume $v_i\geq 0$ in $\R^d$ for all $i=1,\dots,K$. In polar coordinates, suppose by contradiction that $\{v_1=0=v_2\}$ contains an annulus $r \in (a,b)$ (possibly unbounded, i.e. $b=+\infty$). Up to choosing a larger annulus, we can assume that either $a\in \partial\{v_1=0=v_2\}$ or $b\in \partial\{v_1=0=v_2\}$, with $a,b>0$. Suppose $a\in \partial\{v_1=0=v_2\}$ with $a >0$ (the case $b\in \partial\{v_1=0=v_2\}$ with $b>0$ is analogous): then, since $\{v_i>0\}$ is connected, there exists an index in $\{1,2\}$, say $1$, such that $\{v_1>0\}=(r_0,a)$ for some $r_0<a$. Let now $w_1$ be the least energy radial positive solution to
\[
\begin{cases} 
-\Delta v+\lambda_1 v=\mu_1 v^{Kq-1}, \quad v>0 & \text{in $B_{b} \setminus \overline{B_{r_0}}$} \\
v=0 & \text{on $\pa (B_{b} \setminus \overline{B_{r_0}})$}
\end{cases}
\]
By Lemma \ref{lem: monot inclusion}, we deduce that
\[
I_{-\infty}^r(w_1,v_2,\dots,v_K) < I_{-\infty}^r(\mf{v}),
\]
and moreover $(w_1,v_2,\dots,v_K) \in \cM_{-\infty}^r$ (since, even if $w_1>0$ on $(a,b)$, we still have $v_2=0$ there).
Therefore, we obtain a contradiction with the minimality of $\mf{v}$.

\medskip

\textit{Step 3: if $v_i(0)=0$, then $v_i  \equiv 0$ in a neighbourhood of $0$.}

If not, there exists $r_n \to 0^+$ such that $v_i(r_n) >0$. By Step 1, we deduce that there exist $\eps>0$ such that $v_i >0$ in $B_{\eps} \setminus \{0\}$ and $v_i=0$ on $\partial B_\eps$. Then $v_i\in H_0^1(B_\eps\setminus \{0\})$ is a weak solution to
\begin{equation*}
	\begin{cases}
		-\Delta v_i+\lambda_i v_i=\mu_i v_i^{Kq-1}, v_i>0 & \text{in } B_\eps\setminus\{0\} \\
		v_i=0 & \text{on $\partial (B_\eps\setminus\{0\})$}
	\end{cases}
\end{equation*}
Since $H_0^1(B_r \setminus \{0\}) = H_0^1(B_r)$ for every $r$, by the variational characterization of weak solutions and elliptic regularity, $v_i$ is in fact a classical solution to 
\begin{equation*}
	\begin{cases}
		-\Delta v_i+\lambda_i v_i=\mu_i v_i^{Kq-1} & \text{in } B_\eps \\
		v_i\geq0 & \text{in } B_\eps \\
		v_i=0 & \text{on $\partial B_\eps$}
	\end{cases}
\end{equation*}
By the strong maximum principle, since $v_i\neq 0$, we deduce that $v_i>0$ in $B_\eps$, a contradiction.

\medskip

\textit{Step 4: conclusion of the proof.}

By Step 3, we can divide the components of $\mf{v}$ in two non-empty classes: the ones positive at $0$, and the ones vanishing in a neighbourhood of $0$. For concreteness, let $v_1 \equiv 0$ in $B_\eps$ for some $\eps>0$. 
First, by step 2, we have that $v_j(0)>0$ for every $j \in \{2,\dots,K\}$. Moreover, since $v_1 \not \equiv 0$, in order to satisfy the partial segregation condition there exists another index, say $i=2$, such that $v_2$ vanishes somewhere. Let
\[
R:=\inf\left\{r>0: \ v_2(r) = 0\right\} <+\infty.
\]
We have $v_2>0$ in $B_R$ and $v_2=0$ on $\pa B_R$, and hence, by step 1, $v_2 \equiv 0$ in $\R^d \setminus B_R$. We claim that $v_1 \equiv 0$ in $\overline{B_R}$ and $v_1 >0$ in $\R^d \setminus B_R$. Indeed, notice first that, by Steps 1 and 2, $\{v_1>0\}=\R^d\setminus \overline{B_\rho}$ for some $\rho\leq R$. If $\rho<R$, there should exist another index, say $v_3$, such that $v_3(\rho)=0$. But then, by Step 1 again, $v_3 \equiv 0$ in $\R^d \setminus B_\rho$, and $v_2$ and $v_3$ vanish together on a set with non-empty interior. This is in contradiction with Step 2, and proves the previous claim.

At this point, to sum up, we have that $v_2>0$ in $B_R$ and $v_2 \equiv 0$ in $\R^d \setminus B_R$, while $v_1 \equiv 0$ in $\overline{B_R}$ and $v_1 >0$ in $\R^d \setminus \overline{B_R}$. Namely, $v_1$ and $v_2$ satisfy the full segregation condition $v_1 \, v_2 \equiv 0$ in $\R^d$. Therefore, $v_1 \, v_2 \cdots v_K \equiv 0$ in $\R^d$, and it is clear that each $v_j$ with $j=3,\dots,K$ has to minimize the energy
\[
\int_{\R^d} \frac12(|\nabla v|^2+ \lambda_j v^2)-\frac{\mu_j}{Kq}|v|^{Kq}
\]
under the constraints that $v \neq 0$ and $\|v\|_j^2 = |v|_{Kq,j}^{Kq}$; namely (up to a multiplication by $-1$) $v_j$ is a least energy radial positive solution to \eqref{sc pb leps 1}. Furthermore, $(v_1,v_2)$ is a minimizer for the full segregation problem
\[
\inf\left\{ \sum_{i=1}^2 \left(\frac{1}{2} \|u_i\|_i^2 - \frac{1}{Kq} |u_i|_{Kq,i}^{Kq}\right) \left| \begin{array}{l} (u_1,u_2) \in H^1_{\mathrm{rad}}(\R^d, \R^2), \\
u_1\,u_2 \equiv 0 \quad \text{in $\R^d$}\end{array}\right. \right\}.
\]
By well known results on this type of problems (for which we refer to \cite{CoTeVe, NoTaTeVe, TaTe, SoTaTeZi, Bartsch_Willem_1993}), we deduce that $v_1-v_2$ is a least energy radial sign-changing solution to \eqref{lescs}, and this completes the proof.
\end{proof}

\begin{remark}\label{rmk: on d=1}
As anticipated in the introduction, here we discuss the $1$-dimensional case.

All our results for $\beta>0$ can be extended straightforwardly to the case $d=1$. Clearly \Cref{thm:beta grande} holds without any change, whereas the proof \Cref{thm:beta=beta segnato} and \Cref{thm:beta positivo} need a slight adaptation in order to recover some compactness of bounded sequences: for $d \ge 2$, we used the compactness of the embedding $H^1_{\mathrm{rad}}(\R^d)\hookrightarrow L^p(\R^d)$, which fails when $d=1$. It is sufficient to observe that, even in the case $d=1$, any positive solution is radially decreasing, by the symmetry results in \cite{Busca2000,Ikoma_2009}. Indeed, one can check that the result originally stated for $d\geq 2$ in \cite[Theorem 1]{Busca2000} holds also in the 1-dimensional case for systems of the form \eqref{eq:general problem}, $K\geq 3$, as was already noted in \cite[Appendix]{Ikoma_2009} for the case $K=2$, $d=1$. Once the radial monotonicity is established, one can exploit the compactness of the embedding $H^1_{\mathrm{rad,decr}}(\R^d)\hookrightarrow L^p(\R^d)$, where $H^1_{\mathrm{rad,decr}}=\{u\in H^1_{\mathrm{rad}}(\R^d):u \text{ is radially decreasing}\}$, which holds for any $d\geq 1$ (see \cite[Appendix]{Berestycki_Lions_1983}). This way the minimization can be carried out on the smaller set
\begin{equation*}
	\cM_\b^{r,d}=\left\{\mf u\in H^1_{\mathrm{rad,decr}}(\R,\R^K): \ u_i\neq 0, \ \norm{u_i}_i^2=|u_i|_{Kq,i}^{Kq}+\b|\prod_{j=1}^K u_j|_q^q, \quad i=1,\dots,K \right\},
\end{equation*}
still proving the existence of \emph{least energy positive solutions} by what we just pointed out.

When $\b<0$, the 1-dimensional case shows instead, quite surprisingly, different behaviours with respect to the corresponding ones for $d\geq 2$. First of all, exploiting the characterization of $k_{-\infty}^r$ given in \eqref{var car inf}, one can prove, following the same exact argument as in \cite[Theorem 2 (i)]{Mandel2014repulsive}, that
\begin{equation*}
	k_{-\infty}^r=\inf_{\cM_{-\infty}^r}I_{-\infty}=\frac{Kq-2}{2Kq}\left(\sum_{i=1}^K c_{Kq,\lambda_i,\mu_i}^{\frac{Kq}{Kq-2}}+\bar S^{\frac{Kq}{Kq-2}}\right) \quad \text{and it is \emph{not} attained.}
\end{equation*}
We stress the contrast between this result and \Cref{thm:asintotica beta -infinito}. Moreover, it would be interesting to understand whether the other results contained in \cite[Theorem 2 (ii),(iii)]{Mandel2014repulsive} can be recovered in this context as well, that is, the existence of a $\bar\b<0$ such that:
\begin{itemize}
	\item If $\b<\bar\b$, then $k_\b^r=k_{-\infty}^r$ and $k_\b^r$ is \emph{not} attained;
	\item If $\bar\b<\b\leq 0$, then $k_\b^r<k_{-\infty}^r$ and $k_\b^r$ is achieved by a least energy fully non-trivial radial solution to \eqref{eq:general problem}, with all non-negative components.
\end{itemize}
We strongly believe that the existence of $\bar\b$ and the relations between $k_{\b}^r$ and $k_{-\infty}^r$ still hold, since the proof is based on a convenient characterization of $k_\b^r$ similar to \eqref{var car inf} which seems to hold for any number of components $K\geq 2$. As a matter of fact, if $K=3$ one can carry out explicit computations and check the validity of the result, suggesting in a natural way that this can be obtained for any $K$. Despite the intuition, we could not find a rigorous way to deal with an arbitrary number of components $K$ and we leave this as an open problem. For what concerns the existence of a minimizer for $k_\b^r$, some parts of the proof in \cite{Mandel2014repulsive} seem to rely on the binary interaction, hence it is not clear whether this is merely a technical difficulty that can be bypassed, or whether the $K$-wise interaction introduces a substantial difference, leading to a different scenario.
\end{remark}

\bibliographystyle{abbrv}
\bibliography{bibliography}

\end{document}